\documentclass[12pt]{article}
\usepackage{amsmath, amssymb, amsfonts, amsthm}
\usepackage{enumitem}
\usepackage{multicol}
\usepackage{hyperref}
\allowdisplaybreaks
\def \R {{\mathbb{R}}}
\def \Q {{\mathbb{Q}}}
\def \N {{\mathbb{N}}}
\def \Z {{\mathbb{Z}}}
\def \C {{\mathbb{C}}}
\def \F {{\mathbb{F}}}
\def \K {{\mathbb{K}}}
\def \L {{\mathbb{L}}}

\def \S {{\mathcal{S}}}
\def \Od {{\mathcal{O}}}

\def \ep {{\epsilon}}

\newcommand{\sgn}{\operatorname{sgn}}

\newcommand{\Res}{\operatorname{Res}}
\newcommand{\lcm}{\operatorname{lcm}}

\newtheorem*{theorem*}{Theorem}
\newtheorem{theorem}{Theorem}

\newtheorem{conj}[theorem]{Conjecture}

\newtheorem{qu}[theorem]{Question}

\newtheorem*{ex*}{Example}

\title{Survey on irreducibility of trinomials}
\author {Biswajit Koley, A. Satyanarayana Reddy\\
Department of 
Mathematics, Shiv Nadar 
University, India-201314\\ (e-mail: 
bk140@snu.edu.in, satya.a@snu.edu.in).
  }
\date{}
\begin{document}
\maketitle
\begin{abstract}
Let $a,b,c$ be non-zero integers and $f(x)=ax^n+bx^m+c$ be a trinomial of degree $n$. We surveyed the irreducibility criteria of $f(x)$ over rational numbers. 

\end{abstract}
{\bf{AMS(2010)}}: 11R09, 12D05.\\
\section{Introduction}
In modern times polynomial factorization has become a fundamental part of computer algebra systems. The problem of factoring a polynomial into lower degree polynomials has a long and renowned history. D. Knuth \cite{knuth} traces back to Isaac Newton who first attempted to factor a polynomial of arbitrary degree. In {\em Arithmetica Universalis} (1707), Newton gave a method to find linear and quadratic factors of a polynomial with integer coefficients. Since then  several mathematicians attempted to find irreducibility criteria for different kinds of  polynomials and many other tried to find significant properties of their factors when reducible. For example, knowing one of its factors is a cyclotomic polynomial facilitates us to reduce the computational time while checking the reducibility of a family of polynomials. Before defining the cyclotomic polynomial,  recall that the reciprocal of a polynomial $f(x)=a_nx^n+\cdots+a_1x+a_0$ of degree $n$ is defined as $\tilde{f}(x)=x^nf\left(\frac{1}{x}\right)=a_0x^n+a_1x^{n-1}+\cdots+a_n.$ A polynomial $f(x)$ is said to be a {\em reciprocal polynomial} if $f(x)=\pm\tilde{f}(x)$, otherwise it is called {\em non-reciprocal.} It follows from the definition that if $f(x)$ is a reciprocal polynomial, $f(0)\ne 0$, then $f(x)$ is reducible if and only if $\tilde{f}(x)$ is reducible. If $\alpha$ is a nonzero root of a reciprocal polynomial $f(x)$, then $\frac{1}{\alpha}$ is also a root of $f(x)$. One of the widely studied reciprocal families is the family of cyclotomic polynomials. The $n^{\text{th}}$ cyclotomic polynomial is defined as 
$$ \Phi_n(x)=\underset{(j,n)=1}{\prod\limits_{1\le j\le n}} \left(x-e^{\frac{2\pi i j}{n}}\right), i^2=-1. $$

The irreducibility problem of binomials $ax^n+b, a,b\in \Q$ can be converted to the irreducibility of binomials of the form $x^n-a, a\in \Q$. It  has been resolved completely. The criterion is now known as Capelli's theorem due to Capelli \cite{capelli}. 
\begin{theorem}[\cite{capelli}, \cite{redei}, \cite{vahlen}, \cite{biswajit5}]
Let $n\ge 2, a\in \Q$. The polynomial $x^n-a$ is reducible over $\Q$ if and only if either $a=b^t$ for some $t|n, t>1$ or $4|n, a=-4b^4$ for some $b\in \Q$.
\end{theorem}
Naturally, the question arises about the irreducibility criterion for trinomials  $f(x)=ax^n+bx^m+c,$ where $ a\in \N,$ $b,c\in \Z\setminus \{\, 0\,\}$ and $(a,b,c)=1$. Starting from Eisenstein, many irreducibility criteria appeared for different types of trinomials. Yet the question is not fully solved. There is no single criterion which fits to every trinomial. Although several results appeared over time, to the best of our knowledge there is no survey article on this particular topic. Therefore, in this article, we tried to  collect  an exhaustive list of the irreducibility criteria for trinomials available in the literature.

We  compiled the existing irreducibility criteria of trinomials in section~\ref{sec:rt}. In section~\ref{sec:lzt}, we studied the location of zeros of trinomials and their applications in the factorization of trinomials. The discriminant of trinomials was discussed in  section~\ref{sec:dt}. At the end of this article, we ask a few questions, answers to those questions  will enhance the list of families of irreducible trinomials.  Due to the diversity and intricacies of the theory applied in the polynomial factorization, we will concentrate on irreducibility criteria over rational numbers only. Unless specified, throughout the article, we will assume that all the polynomials have content $1$. By this assumption, the irreducibility over $\Q$ and irreducibility over $\Z$ become the same. Further, by $\ep$ or $\ep_i$, we will always mean that they belong to the set $\{\, -1, +1\,\}$.  For in depth study related to polynomial factorizations and for general references on polynomials, the reader should refer to \cite{borweinbook}, \cite{rahman}, and \cite{AS1}.

\section{Reducibility of trinomials}\label{sec:rt}
This section is divided into various subsections depending upon the irreducibility criteria, which apply to a particular type of trinomials. We will first discuss those families, whose reducibility nature has been  resolved completely. Later we proceed to compile the results providing a partial answer to the question of the irreducibility of trinomials. 

\subsection{Reducibility of $\boldsymbol{x^n+b\ep_1x^m+\ep_2}$}
Let $n, m, b\in \N$ and let $f(x)=x^n+b\ep_1x^m+\ep_2$ be a trinomial of degree $n,\ep_1,\ep_2\in \{\, -1, +1\,\}$. The polynomial $f(x)$ is reducible if and only if $\ep_2\tilde{f}(x)=x^n+b\ep_1\ep_2x^{n-m}+\ep_2$ is reducible. Thus, it is sufficient to check the reducibility of  $x^n+b\ep_1x^m+\ep_2$  for $1\le m\le \frac{n}{2}$. 
\subsubsection{Reducibility of $\boldsymbol{x^n+b\ep_1x+\ep_2}$}\label{subsec:m=1}
Perron \cite{perron} is the first mathematician who gave an irreducibility criterion of a polynomial depending upon the magnitude  of the coefficients. He proved that, 

{\bf Perron's criterion \cite{perron}:} Let $f(x)=x^n+a_{n-1}x^{n-1}+\cdots+a_1x+a_0$ be a monic polynomial with integer coefficients and $a_0\ne 0$. If 
\begin{equation*}
|a_{n-1}|>1+|a_{n-2}|+\cdots+|a_1|+|a_0|,
\end{equation*}
then $f(x)$ is irreducible over $\Z$. 

By using Perron's criterion, the polynomial $\ep_2\tilde{f}(x)=x^n+b\ep_1\ep_2x^{n-1}+\ep_2$, and hence $f(x)=x^n+b\ep_1x+\ep_2$ is irreducible for every $b\ge 3$. If $b$ is either $1$ or $2$, then there are polynomials which are reducible. For example, 
\begin{align}
& x^{3}-2x+1=(x-1)(x^2+x-1);\notag\\
& x^{8}+\ep x+1=(x^2+\ep x+1)(x^6-\ep x^5+\ep x^3-x^2+1).\notag
\end{align}
Selmer \cite{selmer} came across the question of reducibility of trinomials of the form $x^n\pm x\pm 1$ while studying continued fractions. With a change of sign in $x$, it is sufficient to consider $x^n+x+1$ and $x^n-x-1$ only. For example, if $n$ is odd and $f(x)=x^n-x+1$, then $-f(-x)=x^n-x-1.$ Hence, if $n$ is odd, then $f(x)$ is reducible if and only if $-f(-x)$ is reducible and so on. 

If $\alpha_i$ is a root of $f(x)=x^n+\ep (x+1)$, then Selmer considered the following sum 
\begin{equation*}
S_f=\sum\limits_{i=1}^n \left(\alpha_i-\frac{1}{\alpha_i}\right).
\end{equation*}  
Since $S_f$ is a symmetric function of roots of $f$, $S_f\in \Q,$ and $S_f\in \Z$ for $f(x)=x^n+\ep (x+1)$. As any factor of $f(x)$ over $\Z$ has a constant term $\pm 1,$ reducibility of $f(x)$ over $\Z$ yields an integer partition of $S_f$. Using this fact, Selmer proved that 

\begin{theorem}[Selmer \cite{selmer}]\label{selmer1}
The polynomial $x^n+\ep (x+1)$ is irreducible except for $\ep =1$ and $n\equiv 2\pmod{3}$. If $\ep=1$ and $n\equiv 2\pmod{3}$, then $x^n+x+1=\Phi_3(x)g(x)$, where $g(x)$ is either identically $1$ or an irreducible polynomial.  
\end{theorem}

Selmer attempted to  answer the reducibility problem of trinomials $x^n+\ep_1 x^m+\ep_2$ for arbitrary values of $m$. By considering the sum  
\begin{equation*}
S_f=\sum\limits_{i=1}^n (\alpha_i^m-\frac{1}{\alpha_i^m}),
\end{equation*}
where $\alpha_i$ is a root of $x^n+\ep_1 x^m+\ep_2$, he gave an upper bound for the number of irreducible factors of it. 

\begin{theorem}[Selmer \cite{selmer}]\label{selmer2}
Let $n=dn_1, m=dm_1, (n_1,m_1)=1$ and $f(x)=x^n+\ep_1 x^m+\ep_2 $ be a polynomial of degree $n\ge 2m$. Then all the roots of $f(x)$ which lie on the unit circle are roots of unity and $\Phi_3(x^d)$ divides $f(x)$ only when $n_1+m_1\equiv 0\pmod{3}.$ Apart from cyclotomic factors, $f(x)$ has at most $m$ irreducible factors, each is of degree greater than $5$ for $n>7$.
\end{theorem}

If $b=2$, then a partial answer can be drawn from the work of Schinzel \cite{AS3} on the trinomial $x^n-2x^m+1$. However for complete answer one can refer either the works of Harrington \cite{harrington} (see Theorem \ref{harrington} below) or by that of Filaseta et al. \cite{filaseta3} (see Theorem \ref{filasetatrinomial} below). The trinomial  $x^n+2\ep_1x+\ep_2$ factors as $(x\pm 1)g(x)$ for some irreducible polynomial $g(x)\in \Z[x]$ and $n\ge 3$. If $n=2,$ then $x^2+2\ep_1x+\ep_2$ is either irreducible or has a cyclotomic factor. Combiningly we have

\begin{theorem}
Let $b\in \N$ and $f(x)=x^n+b\ep_1x+\ep_2$ be a trinomial of degree $n\ge 2$. Then the followings are true.
\begin{enumerate}[label=(\alph*)]
\item If $b\ge 3$, then $f(x)$ is irreducible.
\item If $b=1, 2$, then $f(x)$ is reducible if and only if it factor as $f_c(x)f_n(x)$, where $f_c(x)$ is a product of cyclotomic polynomials, and $f_n(x)\in \Z[x]$ is either identically $1$ or an irreducible polynomial. Further in reducible cases, if $b=1$, $f_c(x)$ is either $\Phi_3(x)$ or $\Phi_6(x)$, and if $b=2$, then $f_c(x)$ is either $x+1, x-1, (x+1)^2$ or $(x-1)^2$. 
\end{enumerate}
\end{theorem}

\subsubsection{Reducibility of $\boldsymbol{x^n+b\ep_1x^m+\ep_2, \, m\ge 2 }$}
\begin{enumerate}[label=(\Alph*)] 
\item $\boldsymbol{b\le 2}$\\
Let $f(x)=x^n+b\ep_1x^m+\ep_2$ be a trinomial of degree $n\ge 2m$  and $b$ is either $1$ or $2,\ep_1, \ep_2\in \{\, -1, +1\,\}$. 
One can get an upper bound on  the number of irreducible factors of $f(x)=x^n+\ep_1 x^m+\ep_2$ from Theorem \ref{selmer2}. In 1960, Tverberg \cite{tverberg} considered Newton's formula for power sum of roots and proved a necessary and sufficient condition for the reducibility of $f$. 

\begin{theorem}[Tverberg \cite{tverberg}]
Suppose $f(x)=x^n+\ep_1x^m+\ep_2$ is a polynomial of degree $n\ge 2m$. Then $f(x)$ is irreducible if and only if $f(x)$ has no root on the unit circle. If $f(x)$ has a root on the unit circle, then it is the roots of unity and the other factor of $f(x)$ is irreducible. 
\end{theorem}

In the same year,  Ljunggren \cite{ljunggren} developed an ingenious tool to study the behavior of the factors of a polynomial $f(x)$ having a small Euclidean norm. If $f(x)=a_nx^n+a_{n-1}x^{n-1}+\cdots+a_1x+a_0$, then the {\em Euclidean norm} of $f$, denoted by $\|f\|_2,$ is defined as 
\begin{equation*}
\|f\|_2=\left( a_n^2+a_{n-1}^2+\cdots+a_1^2+a_0^2\right)^{1/2}.
\end{equation*}

Suppose $f(x)=f_1(x)f_2(x)$ is a non-trivial factorization of $f(x)$. Ljunggren considered the polynomials $g(x)=\tilde{f}_1(x)f_2(x)$ and $\tilde{g}(x)$ so that $g(x)\tilde{g}(x)=f(x)\tilde{f}(x)$. If the value of $\|f\|_2^2$ is small, then one can effectively compute the nature of the factors of $f$. It is interesting to note that $\|f\|^2_2$ is the coefficient of $x^n$ in $f(x)\tilde{f}(x)$. 

Ljunggren applied this method to study trinomials and quadrinomials of the form $x^n+\ep_1x^m+\ep_2$ and $x^n+\ep_1x^m+\ep_2x^r+\ep_3$,  respectively. For several other applications of this method, readers can refer  \cite{boyd}, \cite{smyth}, \cite{dubickas},  \cite{charles}, \cite{flammang}, \cite{grytczuk}. In the case of trinomials, Ljunggren proved that 

\begin{theorem}[Ljunggren \cite{ljunggren}]\label{ljunggrentrinomial}
If $n=n_1d, m=m_1d, (n_1,m_1)=1, n\ge 2m,$ then the polynomial
\begin{equation*}
f(x)=x^n+\ep_1x^m+\ep_2,
\end{equation*}
is irreducible, apart from the following three cases, where $n_1+m_1\equiv 0\pmod{3}$: 
\begin{center}
 $n_1,m_1$ both odd, $\ep_1=1$; $n_1$ even, $\ep_2=1$; $m_1$ even, $\ep_1=\ep_2$.
\end{center}
In these three cases, $f(x)$ is the product of the polynomial $x^{2d}+\ep_1^m\ep_2^nx^d+1$ and a second irreducible polynomial.
\end{theorem}

However, the above  factor is not correct in all cases. For example, by Theorem \ref{ljunggrentrinomial}, the polynomial $x^{50}-x^4-1$ has  factor $x^4+x^2+1$ but they are divisible by $x^4-x^2+1$. More generally, if $d$ is even, $d_1\equiv 5\pmod{6}$, $d_2\equiv 1\pmod{6}$ and $d_3$ is odd, then the polynomials $x^{dd_1}+x^d-1, x^{dd_2}-x^{2d}-1$ and $x^{2^{d_3}d}-x^d+1$ are divisible by $x^{2d}+x^d+1$ by Theorem \ref{ljunggrentrinomial}. But they are divisible by $x^{2d}-x^d+1$. One can check that Ljunggren's result is true only when $\ep_1=\ep_2=1$. We corrected these errors in \cite{biswajit2} and proved that

\begin{theorem}[Koley \& Reddy \cite{biswajit2}]\label{biswajit2trinomial}
Let $n,m\in \N,$ $n\ge 2m$ and $m=2^a\cdot 3^b\cdot m_1,$ $n-2m=2^p\cdot 3^q\cdot n_1,$ where $a, b, p, q\ge 0, (m_1n_1, 6)=1.$ If $\ep_1+\ep_2\ne 2$ and $f(x)=x^n+\ep_1x^m+\ep_2$ is reducible, then $q>b$ and one of the following holds: 
\begin{enumerate}[label=(\alph*)]
\item $\ep_1=-\ep_2=1, a=p$, 
\item $\ep_1=-1, a\ep_2<p\ep_2$.
\end{enumerate}
In each of these cases, the cyclotomic part of $f$ is $f_c(x)=\Phi_6(x^{(n,m)})$.  
\end{theorem}

Schinzel \cite{AS3} studied the polynomials of the form $x^n-2x^m+1$. If $n=2m,$ then $x^n-2x^m+1=(x^m-1)^2$ is a product of cyclotomic polynomials. For $n\ne 2m$, he proved that 
\begin{theorem}[Schinzel \cite{AS3}]\label{schinzelas3}
Let $n\ne 2m$. The polynomial $x^n-2x^m+1$ can be factored as $(x^{(n,m)}-1)g(x)$ where $g(x)\in \Z[x]$ is an irreducible non-reciprocal polynomial except in the following cases:
\begin{align}
x^{7k}-2x^{2k}+1=(x^k-1)(x^{3k}+x^{2k}+1)(x^{3k}+x^k+1); \notag\\
x^{7k}-2x^{5k}+1=(x^k-1)(x^{3k}+x^{2k}+1)(x^{3k}-x^k-1),\notag
\end{align}
for every $k\ge 1$. 
\end{theorem}

If $n=2m,$ then $x^n+2x^m+1=(x^m+1)^2$ is a product of cyclotomic polynomials. On the other hand, if we replace $x$ by $-x$ in Theorem \ref{schinzelas3}, then one can conclude the following. 

{\em If $n$ is even, $m$ is odd, then $x^n+2x^m+1$ can be factored as $(x^{(n,m)}+1)g(x)$ where $g(x)$ is an irreducible non-reciprocal polynomial. }

Similarly, partial results can be drawn for $x^n+2x^m-1$ and $x^n-2x^m-1$ by replacing $x$ with $-x$ in  Theorem \ref{schinzelas3}. Recently Filaseta et al. \cite{filaseta3} gave a complete characterization of the irreducibility of $x^n+2\ep_1x^m+\ep_2$. 

\begin{theorem}[Filaseta et al. \cite{filaseta3}]\label{filasetatrinomial}
Let $n>m>0$ and $f(x)=x^n+2\ep_1x^m+\ep_2$. Then the non-cyclotomic part of $f(x)$ is reducible in the following cases:
\begin{align}
    & x^{7k}-2x^{2k}+1=(x^k-1)(x^{3k}+x^{2k}-1)(x^{3k}+x^k+1);\notag\\
    & x^{7k}-2x^{5k}+1=(x^k-1)(x^{3k}+x^{2k}+1)(x^{3k}-x^k-1);\notag\\
    & x^{7k}+2x^{2k}-1=(x^k+1)(x^{3k}-x^{2k}+1)(x^{3k}+x^k-1);\notag\\
    & x^{7k}-2x^{5k}-1=(x^k+1)(x^{3k}-x^k+1)(x^{3k}-x^k-1);\notag\\
    & x^{7k}+2x^{3k}+1=(x^{3k}-x^{2k}+1)(x^{4k}+x^{3k}+x^{2k}+1);\notag\\
    & x^{7k}+2x^{4k}+1=(x^{3k}-x^{k}+1)(x^{4k}+x^{2k}+x^k+1);\notag\\
    & x^{7k}+2x^{3k}-1=(x^{3k}+x^{2k}-1)(x^{4k}-x^{3k}+x^{2k}+1);\notag\\
    & x^{7k}-2x^{4k}-1=(x^{3k}-x^k-1)(x^{4k}+x^{2k}-x^k+1),\notag
\end{align}
for every $k\ge 1.$
\end{theorem}

\item $\boldsymbol{b\ge 3}$

Let $f(x)=x^n+b\ep_1x^m+\ep_2$ be a polynomial of degree $n\ge 2m$ and  let $b\ge 3,\ep_1,\ep_2\in \{\, -1, +1\,\}$. Recall from the subsection~\ref{subsec:m=1} that these polynomials are irreducible when $m=1$.  Minkusinski and Schinzel \cite{JS} considered trinomials of the form $x^n+p\ep_1x^m+\ep_2,$ where $p$ is an odd prime. They proved that there are only finitely many such trinomials that are reducible.

\begin{theorem}[Minkusinski \& Schinzel \cite{JS}]\label{minkusinski}
If $p$ is an odd prime, then there are only a finite number of ratios $\frac{n}{m}$ for which $f(x)=x^n+p\ep_1x^m+\ep_2$ is reducible. 
\end{theorem}

In 1971,  Schinzel \cite{AS4} asked the following question:\\
{\em Is it possible to find a reducible trinomial of the form $x^n+Ax^m+1$ with $|A|>2, n\ne 2m$ and $n>m>0$?}

Coray answered this question in a letter to Schinzel furnishing the following examples:
\begin{align}
& x^{13}-3x^4+1=(x^3-x^2+1)(x^{10}+x^9+x^8-x^6-\cdots +x^2+1);\notag\\
& x^{13}+3x^7+1=(x^4-x^3+1)(x^9+x^8+x^7+x^6-x^4+x^3+1).\notag
\end{align}

A. Bremner \cite{bremner} considered the trinomials $x^n+Ax^m+1$ with $|A|>2$ and proved more than what Schinzel asked. He showed that there are only finitely many trinomials with an irreducible cubic factor. More precisely, 

\begin{theorem}[Bremner \cite{bremner}]
Let $n\ge 2m>0$ and the trinomial $x^n+Ax^m+1$ has an irreducible factor of degree three, $0\ne A\in \Z$. If $n>3$, then the only possibilities are the following:
\begin{align}
& x^4+2\ep x+1=(x+\ep)(x^3-\ep x^2+x+\ep); \notag\\
& x^5+x+1=(x^2+x+1)(x^3-x^2+1);\notag\\
& x^7-2x^2+1=(x-1)(x^3+x^2-1)(x^3+x+1);\notag\\
& x^7+2x^3+1=(x^3-x^2+1)(x^4+x^3+x^2+1);\notag\\
& x^{13}-3x^4+1=(x^3-x^2+1)(x^{10}+x^9+x^8-x^6-2x^5-2x^4-x^3+x^2+1);\notag\\
& x^{33}+67x^{11}+1=(x^3+x+1)(x^{30}-x^{28}-x^{27}+x^{26}+2x^{25}-\cdots-2x^3+x^2-x+1).\notag
\end{align} 
\end{theorem}

If $n$ is odd, then by changing the variable $x$ by $-x$, the polynomial  $x^n+Ax^m-1$ has irreducible cubic factors in the above cases only. If $n$ is even, Lutczyk \cite{AS4} provided the following example, 
\begin{equation*}
x^8+3x^3-1=(x^3+x-1)(x^5-x^3+x^2+x+1).
\end{equation*}
Bremner \cite{bremner} also gave several examples when $n$ is even. For example, 
\begin{equation*}
x^6+(4b^4-4b)x^2-1=(x^3+2bx^2+2b^2x+1)(x^3-2bx^2+2b^2x-1),
\end{equation*}
where $b\ne 0,1, b\in \Z$ is an infinite family of polynomials having a cubic factor. By Perron's criterion, $x^3+2b^2\ep_1x^2+2bx+\ep_2$ is  irreducible when $b\ne 0, 1$ and hence so is its reciprocal.  However, it is not known whether there are only finitely many such  trinomials of the form $x^n+Ax^m-1$ having a cubic factor when $n$ is even. We, therefore, ask the following question, 

\begin{qu}
Let $n,m,A\in \N$, $n\ge 2m,$ $A> 2$ and let $n$ be an even integer.  Are there finitely many trinomials of the form $x^n\pm Ax^m-1$ which have an irreducible cubic factor? 
\end{qu}

\end{enumerate}

\subsection{Reducibility of $\boldsymbol{x^n+b\ep_1x^m+c\ep_2, \, c>1}$}
Suppose $f(x)=x^n+b\ep_1x^m+c\ep_2$ is a trinomial of degree $n$, where $b,c\in \N, c>1,\ep_1,\ep_2\in \{\, -1, +1\,\}$. Unlike in the case of $x^n+b\ep_1x^m+\ep_2$, here $\ep_2 \tilde{f}(x)$ is not of the form of $f(x)$. So, we have to consider the reducibility of $f(x)$ for $1\le m\le n-1$. Nagell \cite{nagell} gave the following irreducibility criterion for trinomials with a large middle coefficient.  

\begin{theorem}[Nagell \cite{nagell}]\label{nagell}
Let $b,c\in \N$ and let $f(x)=x^n+b\ep_1x^m+c\ep_2$ be a polynomial of degree $n$. Then $f(x)$ is irreducible if 
\begin{enumerate}[label=(\alph*)]
\item $b>1+c^{n-1}$ and
\item if $d|n, d>1,$ then $c$ is not a $d^{\text{th}}$ power. In particular, $c>1$. 
\end{enumerate}
\end{theorem}

Nagell's criterion is weaker than that of Perron's criterion in the case of $m=n-1$. If $m=n-1,$ then the second condition of Nagell's theorem becomes redundant.

\subsubsection{Reducibility of $\boldsymbol{x^n+b\ep_1x^{n-1}+c\ep_2,\, c>1}$}

 J. Harrington \cite{harrington} asked the values of $n$ and $b$ for which   $$f(x)=x^n+b(x^{n-1}+x^{n-2}+\cdots+x+1)$$ 
 is irreducible. If $b=1,$ then $f(x)$ is a product of cyclotomic polynomials. For $b=-1$, one can see from Theorem \ref{schinzelas3} and  

\begin{equation*}
x^n-2x^{n-1}+1=(x-1)(x^{n-1}-x^{n-2}-\cdots-x-1),
\end{equation*}
that $f(x)$ is irreducible. However, the answer is not known for $|b|\ge 2$. Harrington studied the trinomials of the form $x^n+b\ep_1x^{n-1}+c\ep_2$  to answer these questions and proved that 

\begin{theorem}[Harrington \cite{harrington}]\label{harrington}
Let $n\ge 3, b,c\in \N, nb\ne 9, b\ne c, b\ge 2$ and $c\le 2(b-1)$. If the trinomial $f(x)=x^n+b\ep_1 x^{n-1}+c\ep_2$ is reducible, then $f(x)=(x\pm 1)f_n(x),$ where $f_n(x)\in \Z[x]$ is a non-reciprocal irreducible polynomial. 
\end{theorem} 
Theorem \ref{harrington} is no longer true if $c>2(b-1)$. For example,
\begin{align}
& x^6+2\ep x^5+4=(x^3-2x+2\ep)(x^3+2\ep x^2+2x+2\ep);\notag\\
& x^7+2x^6+9=(x^3-3x+3)(x^4+2x^3+3x^2+3x+3).\notag
\end{align}
By Perron's criterion, if $c<b-1,$ then $x^n+b\ep_1x^{n-1}+c\ep_2$ is irreducible. Hence, 
\begin{theorem}\label{harringtonsummary}
Let $n\ge 4, b\ge 2, c\in \N,$ and let $f(x)=x^n+b\ep_1x^{n-1}+c\ep_2$ be a polynomial of degree $n$.
\begin{enumerate}[label=(\alph*)]
    \item If $c<b-1,$ then $f(x)$ is irreducible.
    \item If $b-1\le c\le 2(b-1), b\ne c$ and $f(x)$ is reducible, then $f(x)=(x\pm 1)f_n(x),$ where $f_n(x)\in \Z[x]$ is a non-reciprocal irreducible polynomial. 
\end{enumerate}
\end{theorem}

 If $n$ is even, then $x^n+bx^{n-1}+b-1=(x+1)f_n(x)$ shows that the lower bound of $c$ can not be improved further. For $b=c,$ Harrington conjectured that
\begin{conj}[Harrington \cite{harrington}]\label{harringtonconjecture}
Let $n,b\in \N$, $b\ge 2$ and $f(x)=x^n+b\ep_1x^{n-1}+b\ep_2$. If $f(x)\ne x^2+4x +4$ and $f(x)\ne x^2-4x+4,$ then $f(x)$ is irreducible. 
\end{conj}

\subsubsection{Reducibility of $\boldsymbol{x^{2m}+b\ep_1x^m+c\ep_2, \, c>1}$}
Let $b,c\in \N, c>1$ and $f(x)=x^{2m}+b\ep_1x^m+c\ep_2$ be a polynomial of degree $2m$. From Nagell's criterion, $f(x)$ is irreducible whenever $b>1+c^{2m-1}$ and $c\ne u^d,$ where $d|2m, d>1, u\in \N$. Also note that if $b^2-4c\ep_2$ is a square, then $x^2+b\ep_1x+c\ep_2$ is reducible and hence $x^{2m}+b\ep_1x^m+c\ep_2$ is a product of two $m$ degree polynomials. 
 
A. Schinzel studied the problem of reducibility of trinomials extensively in \cite{AS9}, \cite{AS10}, \cite{AS11}. He even considered the factorization problem of trinomials over an arbitrary field and provided many necessary and sufficient conditions for several families of trinomials. One such criterion is 
\begin{theorem}[Schinzel \cite{AS9}]
Let $K$ be any field of characteristic different from $2$ and $b,c\in \N$. Then the polynomial $x^{2m}+b\ep_1x^m+c\ep_2$ is reducible over $K$ if and only if either $b^2-4c\ep_2$ is a square in $K$ or there is a prime $p$ dividing $m$ such that $x^{2p}+b\ep_1x^p+c\ep_2$ is reducible over $K$ or $4|m$ and $x^8+b\ep_1x^4+c\ep_2$ is reducible over $K$.
\end{theorem}

It is, therefore, sufficient to find the reducibility condition of $x^{2p}+b\ep_1x^p+c\ep_2$ to completely determine the reducibility of $x^{2m}+b\ep_1x^m+c\ep_2,$ where $p$ is a prime dividing $m$. In the course of  determining a Galois group of a family of polynomials, Filaseta et al. \cite{filaseta8} proved that 

\begin{theorem}[Filaseta et al. \cite{filaseta8}]
Let $p$ be an odd prime and $u\ge 2$. Then $x^{2p}+\ep_1x^p+u^p$ is irreducible over $\Q$.
\end{theorem} 

Later they separately considered the reducibility of trinomials of the form $x^{2p}+b\ep_1x^p+c\ep_2,$ where $b, c \in \N $ and $p$ is a prime number in \cite{filaseta2}. They noticed that if $b^2-4c\ep_2=d^2$ or $c=b^2=d^{2p}, p\ge 5$ or $c=\frac{b^2}{2}=2^pd^{2p}, p\ge 3$ or $c=\frac{b^2}{3}=3^pd^{2p}, p\ge 5,$ for some integer $d\in \Z$, then $x^{2p}+b\ep_1x^p+c\ep_2$ is reducible over $\Z$. It has been proved in \cite{filaseta2} that if $b,c$ do not satisfy any of these conditions, then $x^{2p}+b\ep_1x^p+c\ep_2$ is irreducible over $\Z$ provided 
\begin{equation*}
p\nmid (181^2-1)\underset{q|b}{\prod\limits_{q \text{ prime }}} (q^2-1).
\end{equation*}
However, there are irreducible polynomials even when $p$ divides the above product. For example, $x^4+2x^2+3$ is irreducible. Here, $p=b=2, c=3;$ $b,c$ do not satisfy any of the above conditions and the product is $(181^2-1)(2^2-1)$ divisible by $2.$

\subsubsection{Reducibility of $\boldsymbol{x^n+b\ep_1x+c\ep_2, \, c>1}$}
Suppose $f(x)=x^n+b\ep_1x+c\ep_2$ is a polynomial of degree $n\ge 2,$ where $b\in \N$ and  $c\ge 2, \ep_1,\ep_2\in \{\, -1, +1\,\}$.   
Serret \cite{serret} and Ore \cite{ore} studied the irreducibility of trinomials of the form $x^n-x+c\ep_2$ independently. By the use of reducibility of polynomials modulo a prime, they gave an irreducibility criterion for them.  

\begin{theorem}[Serret \cite{serret}, Ore \cite{ore}]
Let $c\in \N$ and $f(x)=x^p-x+c\ep_2$ be a polynomial of prime degree $p$. Then $f(x)$ is irreducible if $p\nmid c$. 
\end{theorem}
There are irreducible polynomials with $p|c$. For example, $x^5-x+10$ and $x^2-x+2$ both are irreducible. Irreducible polynomials can be found for  composite values of $n$ as well. For example, 
\begin{align}
& x^4-x+2; \, \, && x^{15}-x+3;\notag\\
& x^{10}-x+6;\,\, && x^6-x+18,\notag
\end{align} 
are all irreducible. A.I. Bonciocat and N.C. Bonciocat \cite{bonciocat1} gave an irreducibility criterion for arbitrary polynomials with constant term divisible by large prime power. A version of their result in case of trinomials is 
\begin{theorem}[Bonciocat \& Bonciocat \cite{bonciocat1}]\label{bonciocattrinomial}
Let $a,b,c\in \N,$ $p$ be a prime number, and $p\nmid bc$, $r\ge 0, k\ge 1$. The polynomial $ax^n+bp^r\ep_1x+p^kc\ep_2$ is irreducible if 
\begin{equation*}
p^k>bp^{2r}+c^{n-1}ap^{nr}.
\end{equation*}
\end{theorem}

It is not possible to improve Theorem \ref{bonciocattrinomial} for arbitrary values of $a,b,c$. For example, let $f(x)=x^2+2x-8$. Then $a=b=c=1, p=n=2, r=1, k=3, 2^k=2^{2r}+2^{nr}$ and $f(x)=(x-2)(x+4)$.

\subsubsection*{{\bf Reducibility} of $\boldsymbol{x^n+p^r\ep_1x+p^k\ep_2}$}

Let $p$ be a prime number and $f(x)=x^n+p^r\ep_1x+p^k\ep_2$ be a trinomial of degree $n\ge 2,$ $ r, k\ge 0, \ep_1, \ep_2\in \{\, -1, +1\,\}$.
If $a=b=c=1$ in Theorem \ref{bonciocattrinomial}, then 

{\em The polynomial $x^n+p^r\ep_1 x+p^k\ep_2$ is irreducible if}\\ $$p^k>p^{2r}+p^{nr},\quad\text{ where }  k\ge 1, r\ge 0.$$ 

There are irreducible polynomials that do not satisfy the condition of Theorem \ref{bonciocattrinomial}. For example, $x^4+2\ep_1 x+2^4\ep_2$ is irreducible while $2^4<2^2+2^4$.

It is possible to find more irreducible families of polynomials of the form $x^n+p^r\ep_1x+p^k\ep_2$, where $p$ is a prime number and $r,k\ge 0$.  The case $k=0$ has already been solved in $x^n+b\ep_1x+\ep_2$ and $k=1, r\ge 1$ case follows from Eisenstein's criterion.

{\bf Eisenstein's criterion \cite{eisen}:} Let $p$ be a prime number and let $f(x)=a_nx^n+\cdots+a_1x+a_0\in \Z[x]$ be a polynomial of degree $n$. If $p|a_i$ for $0\le i\le n-1$, and $p\nmid a_n, p^2\nmid a_0$, then $f(x)$ is irreducible over $\Z$. 

 The case $k=1, r=0$ and $p$ is an odd prime follows from Theorem \ref{bonciocattrinomial}. In \cite{biswajit3} (see Theorem \ref{ourprimecase} below), we have shown that if $x^n+\ep_1 x+2\ep_2$ is reducible, then it is a product of cyclotomic factors and an irreducible non-reciprocal polynomial.  If $k\ge 2$ and $r=0$, then it is irreducible by Theorem \ref{bonciocattrinomial}. An alternate proof of this particular case can also be found in \cite{biswajit2}. For $k\ge 2, r=1$, we will show that 
\begin{theorem}
Let $p$ be a prime number and $k\ge 1, n\ge 2$. Then $x^n+p\ep_1x+p^k\ep_2$ is irreducible. 
\end{theorem}
\begin{proof}
Suppose $f(x)$ is reducible. Let $f(x)=f_1(x)f_2(x)$ be a proper factorization of $f(x)$, where 
\begin{align}
& f_1(x)=x^s+c_1x^{s-1}+\cdots+c_{s-1}x+c_s;\notag\\
& f_2(x)=x^t+d_1x^{t-1}+\cdots+d_{t-1}x+d_t,  \quad t+s=n.\notag
\end{align}
Comparing the coefficient of $x$ in $f(x)=f_1(x)f_2(x)$, we have
 \begin{align}
& c_{s-1}d_t+c_sd_{t-1}=p\ep_1,\label{trinomialeq1}
\end{align}
Reducing the coefficient modulo $p$, we get $f_p(x)=x^n=(f_1)_p(x)(f_2)_p(x)$. This is possible only when $p|c_i, d_j$ for every $i$ and $j$. In other words, $p^2$ divides the left hand side of \eqref{trinomialeq1}, which is not the case for the right hand side of \eqref{trinomialeq1}. Hence, $x^n+p\ep_1x+p^k\ep_2$ is irreducible for every $k\ge 1, n\ge 2$.
\end{proof}

If $k\ge 2$ and $r\ge 2$, then there are polynomials which are reducible. For example, 
\begin{align}
& x^5-9x-27=(x^2+3x+3)(x^3-3x^2+6x-9);\notag\\
& x^5+81x+243=(x^2+3x+9)(x^3-3x^2+27);\notag\\
& x^8+81x-81=(x^2-3x+3)(x^6+3x^5+6x^4+9x^3+9x^2-27).\notag
\end{align}
The complete characterization of the case $k\ge 2$ and $r\ge 2$ is not known. The condition of Theorem \ref{bonciocattrinomial}, $p^k>p^{2r}+p^{nr}$ is equivalent to $k>nr$ except for $k= nr+1, p= 2, n= 2$. If $p=2, n=2$ and $k=nr+1$, then $f(x)=x^2+2^r\ep_1x+2^{2r+1}\ep_2$ is reducible only when $\ep_2=-1,$ and in that case $f(x)=(x+2^{r-1}(3+\ep_1))(x-2^{r-1}(3-\ep_1))$. 

On the other hand, the polynomial $f(x)=x^n+p^r\ep_1x+p^k\ep_2$ is irreducible over $\Z$ if and only if $f(p^kx)=p^k\ep_2(p^{k(n-1)}\ep_2x^n+p^r\ep_1\ep_2x+1)$ is irreducible over $\Q$, that is,  if and only if $p^{k(n-1)}\ep_2x^n+p^r\ep_1\ep_2x+1$ is irreducible over $\Z$. The polynomial $\tilde{g}(x)=p^{k(n-1)}\ep_2x^n+p^r\ep_1\ep_2x+1$ is irreducible over $\Z$ if and only if $g(x)=x^n+p^r\ep_1\ep_2x^{n-1}+p^{k(n-1)}\ep_2$ is irreducible over $\Z$. From Perron's criterion, the polynomial $g(x)$ is irreducible if $p^r>p^{k(n-1)}+1$. Since $k\ge 2$, this is same as $r>k(n-1)$.

Therefore, if we summarize all these results for $x^n+p^r\ep_1x+p^k\ep_2$, then it can be put in the following theorem and table below. 
 
\begin{theorem}\label{primepower1}
Let $p$ be a prime number, $k,r\ge 0$ and let $f(x)=x^n+p^r\ep_1x+p^k\ep_2$ be a polynomial of degree $n\ge 2$.
\begin{enumerate}[label=(\alph*)]
\item If $f(x)$ is reducible and either $r\in\{\, 0, 1\,\}$ or $k\in \{\, 0,1\,\}$, then $f(x)=f_c(x)f_n(x)$, where $f_c(x)$ is a product of cyclotomic polynomials and $f_n(x)$ is either $1$ or a non-reciprocal irreducible polynomial. 
\item If $k\ge 2, r\ge 2$ and $f(x)\ne x^2+2^r\ep_1x-2^{2r+1},$ then $f(x)$ is irreducible except possibly for $\frac{k}{n}\le r\le k(n-1)$.
\end{enumerate}  
\end{theorem}

\begin{tabular}{|l|l|l|l|l|}
\hline
$k$ & $r$ & $f(x)$ & Irreducible & Reducible \\[3 pt]
\hline
$0$ & $0$ & $x^n-x-1$ & $n\ge 2$ & \\[3 pt]
\cline{3-5}
& & $x^n+x-1$ & $n\not\equiv 5\pmod{6}$ &  $n\equiv 5\pmod{6},$\\[3 pt]
& & & & $f(x)=\Phi_6(x)f_n(x)$
\\[3 pt]
\cline{3-5}
& & $x^n+\ep x+1$ & $n\not\equiv 2\pmod{6}$ &  $n\equiv 2\pmod{6},$  \\[3 pt]
& & & & $f(x)=\Phi_3(\ep x)f_n(x)$\\[3 pt]
 \hline
$0$ & $\ge 1$ & $x^n+\ep_1 p^rx+\ep_2$& $p^r>2$ & \\[3 pt]
\cline{3-5}
& & $x^n+2x-1 $ & $n\ge 2$ & \\[3 pt]
\cline{3-5}
& & $x^n-2x-1$ & $n\equiv 0\pmod{2}$ & $n\equiv 1\pmod{2},$\\[3 pt]
& & & & $f(x)=(x+1)f_n(x)$\\[3 pt]
\cline{3-5}
& & $ x^n-2x+1$ & & \small{$n\ge 3, f(x)=(x-1)f_n(x)$} \\[3 pt]
 & & & & \small{ $n=2, f(x)=(x-1)^2$}\\[3 pt]
\cline{3-5}
& & $x^n+2x+1$ & $n\equiv 1\pmod{2}$ & $n\equiv 0\pmod{2},$\\[3 pt]
& &  & & \small{$n\ge 3, f(x)=(x+1)f_n(x)$}\\[3 pt]
& & & & \small{$n=2, f(x)=(x+1)^2$}\\[3 pt]
\hline
$1$ & $0$ & $x^n+\ep_1 x+p\ep_2$ & $p\ge 3$ & \\[3 pt]
\cline{3-5}
&& $x^n+x-2$ & & $f(x)=(x-1)f_n(x)$\\[3 pt]
\cline{3-5}
& & $x^n-x+2$ & $n\ge 2$ &\\[3 pt]
\cline{3-5}
&& $x^n+x+2$ & $n\equiv 0\pmod{2}$ & $n\equiv 1\pmod{2},$\\[3 pt]
& & && $f(x)=(x+1)f_n(x)$\\[3 pt]
\cline{3-5}
&& $x^n-x-2$ & $n\equiv 1\pmod{2}$ & $n\equiv 0\pmod{2},$\\[3 pt]
& &&& $f(x)=(x+1)f_n(x)$\\[3 pt]
\hline
$1$ & $\ge 1$ & $x^n+p^r\ep_1x+p^k\ep_2$ & $n\ge 2$ & \\[3 pt]
\hline 
$\ge 2$ & $\le 1$ & $x^n+p^r\ep_1x+p^k\ep_2$ & $n\ge 2$ & \\[3 pt]
\cline{2-5}
& $\ge 2$ & $x^n+p^r\ep_1x+p^k\ep_2$ & \multicolumn{2}{|l|}{According to Theorem \ref{primepower1}} \tabularnewline
\hline
\end{tabular}

\begin{qu}
Let $p$ be a prime number, $r,k,n\ge 2.$ For what values of $r$, where $\frac{k}{n}\le r\le k(n-1),$ the polynomial $x^n+p^r\ep_1x+p^k\ep_2, \ep_1,\ep_2\in \{\, -1, +1\,\}$ is irreducible over $\Z$?
\end{qu}

\subsubsection{{\bf Reducibility of}  $\boldsymbol{x^n+b\ep_1x^m+p^k\ep_2, \, k\ge 1}$}
Instead of restricting the  value of $m$ and the middle coefficient, let $f(x)=x^n+b\ep_1x^m+p^k\ep_2$ be an $n$ degree trinomial, where $p$ is a prime number, $b, k\in \N,\ep_1, \ep_2\in \{\, -1, +1\,\}$. If $k=1$ and $p|b$, then $f(x)$ is irreducible due to Eisenstein's criterion. If $k=1$ and $p\nmid b$, there are only finitely many $b$ values for which $f(x)$ is reducible. On the other hand, for $k\ge 2$, the polynomial is irreducible provided $p^k$ is sufficiently large. 

\begin{enumerate}[label=(\Alph*)]
    \item $\boldsymbol{k=1}$\\

By Nagell's criterion, $x^n+b\ep_1x^m+p\ep_2$ is irreducible whenever $b>p^{n-1}+1$. 
If the constant term is a large prime number, then Panitopol and Stef\"{a}nescu \cite{LP} gave an effective criterion for the irreducibility of polynomials.

\begin{theorem}[Panitopal \& Stef\"anescu \cite{LP}]\label{panitopol1}
Let $f(x)=a_nx^n+\cdots+a_1x+a_0\in \Z[x]$ be a polynomial with $|a_0|>|a_1|+\cdots+|a_n|$. If $|a_0|$ is prime or $\sqrt{|a_0|}-\sqrt{|a_n|}<1$, then $f(x)$ is irreducible in $\Z[x]$. 
\end{theorem}

Thus, $x^n+b\ep_1x^m+p\ep_2$ is irreducible for every prime $p>b+1$. Theorem \ref{panitopol1} fails to answer the irreducibility of $x^n+\ep_1x^m+2\ep_2$. We have considered the equality condition in Theorem \ref{panitopol1} and proved that 

\begin{theorem}[Koley \& Reddy \cite{biswajit3}]\label{ourprimecase}
Let $f(x)=a_{n_r}x^{n_r}+a_{n_{r-1}}x^{n_{r-1}}+\cdots+a_{n_1}x^{n_1}+a_0\in \Z[x], r\ge 2$, where $a_i\ne 0$ for every $i$  and let $|a_{n_1}|+\cdots+|a_{n_{r-1}}|+|a_{n_r}|=|a_0|$  be a prime number. If $f(x)$ is reducible, then it is the product of an irreducible non-reciprocal polynomial and distinct cyclotomic polynomials. Furthermore, the product of all of its cyclotomic factors is given by   $f_c(x)=(x^{n_r}+\sgn(a_0a_{n_r}), x^{n_{r-1}}+\sgn(a_0a_{n_{r-1}}),\ldots, x^{n_1}+\sgn(a_0a_{n_1})),$ where $\sgn(x)$ denotes the sign of $x\in\R$.
\end{theorem}

Theorem  \ref{ourprimecase} is not true if the constant term is not a prime number. For example, 
\begin{align}
& x^4+3x^2+4=(x^2-x+2)(x^2+x+2);\label{trieq1}\\
& 2x^3+11x^2+8x-21=(x-1)(x+3)(2x+7).\notag
\end{align}

From Theorem \ref{ourprimecase}, the polynomial $x^n+(p-1)\ep_1x^m+p\ep_2$ has only one non-reciprocal irreducible factor apart from its cyclotomic factors. By Eisenstein's criterion, if $b\equiv 0\pmod{p}$, then $x^n+b\ep_1x^m+p\ep_2$ is irreducible.   Therefore, combining all of them, one can see that
\begin{theorem}\label{constantprimesummary}
Let $p$ be a prime, $b\in \N$ and let $f(x)=x^n+b\ep_1x^m+p\ep_2$ be a polynomial of degree $n$. Then $f(x)$ is irreducible except possibly for 
\begin{equation*}
p-1\le b\le p^{n-1}+1,\quad  b\not\equiv 0\pmod{p}.
\end{equation*}
If $b=p-1$ and $f(x)$ is reducible, then $f(x)=f_c(x)f_n(x)$, where $f_c(x)=(x^n+\ep_2, x^m+\ep_1\ep_2)$ is the cyclotomic part of $f(x)$ and $f_n(x)\in \Z[x]$ is a non-reciprocal irreducible polynomial. 
\end{theorem}

\item $\boldsymbol{k\ge 2}$\\

 For $m=1$, we have the irreducibility criterion given in Theorem \ref{bonciocattrinomial}. If $m=2$, then A.I. Bonciocat and N.C. Bonciocat provided a criterion similar to Theorem \ref{bonciocattrinomial}. 

\begin{theorem}[Bonciocat \& Bonciocat \cite{bonciocat1}]\label{bonciocattrinomial1}
Let $a,b,c\in \N$ and $p$ be a prime number, $p\nmid bc, r\ge 0, k\ge 1$. The polynomial $f(x)=ax^n+bp^r\ep_1x^2+p^kc\ep_2$ is irreducible if $k$ is odd and
\begin{equation*}
p^k>bcp^{3r}+c^{n-1}ap^{nr}.
\end{equation*}
\end{theorem}
The theorem is not true if $k$ is an even integer. For example, let $f(x)=x^4+7x^2+2^4$. Then $a=1=c, b=7, n=k=4, p=2, r=0$ and $2^4>7+1$. However, $f(x)=(x^2-x+4)(x^2+x+4)$. 

When $k$ is an even integer,  Jonassen \cite{jonassen} studied trinomials of the form $x^n+\ep_1 x^m+4\ep_2$. By applying a refinement of Ljunggren's method, he proved that

\begin{theorem}[Jonassen \cite{jonassen}]\label{jonassen}
Suppose $n>m$ are positive integers. Then the polynomial $x^n+\ep_1 x^m+4\ep_2$ is irreducible except in the following cases:
\begin{align}
x^{3k}+\ep_1x^{2k}+4\ep_1 &=(x^k+2\ep_1)(x^{2k}-\ep_1x^k+2);\notag\\
x^{5k}+\ep_1x^{2k}-4\ep_1 &=(x^{3k}+\ep_1x^{2k}-x^k-2\ep_1)(x^{2k}-\ep_1x^k+2); \notag\\
x^{11k}+\ep_1x^{4k}+4\ep_1 &=(x^{5k}-x^{3k}-\ep_1x^{2k}+2\ep_1)(x^{6k}+x^{4k}+\ep_1x^{3k}+x^{2k}+2),\notag
\end{align}
for any $k\ge 1$ and the factors on the right hand side are irreducible in each cases.
\end{theorem}

If $c=1, r=0$ in both Theorem \ref{bonciocattrinomial} and Theorem \ref{bonciocattrinomial1}, then the conditions read $p^k>a+b$, which is similar to that of Theorem \ref{panitopol1}. Because of this reason, we considered  polynomials of the form $f(x)=a_nx^n+\cdots+a_3x^3+p^k\ep,$ where $p$ is a prime number, $p\nmid a_3a_n$ and $k\ge 2$ in \cite{biswajit4}. To frame the results of \cite{biswajit4} conveniently, we introduce the following sets: Let $n>m>0, a,b\in \N$ and $p$ be a prime number. We define 
$$\S_m=\{\, ax^n+b\ep_1x^m+p^k\ep_2\mid p\nmid ab, k\ge 2, \ep_1,\ep_2\in\{\, -1, +1\,\}\,\}$$

 and for $m\ge 2, \, \S_m'=\{\, f\in \S_m\mid k\not\equiv 0\pmod{m}\,\}$. Then it has been shown that 

\begin{theorem}[Koley \& Reddy \cite{biswajit4}]
Let $f(x)=ax^n+b\ep_1x^3+p^k\ep_2\in \S_3'$ be a trinomial of degree $n\ge 4$. If $p^k>a+b$, then $f(x)$ is irreducible. 
\end{theorem}
The condition $p\nmid ab$ in the definition of $\S_m$ cannot be dropped. For example, 
\begin{equation*}
2x^4-3x^3-81=(x-3)(2x^3+3x^2+9x+27).
\end{equation*}
Further, the necessity of the condition $3\nmid k$ can be seen from the following examples:
\begin{align}
& x^4+4\ep x^3+27=(x+3\ep)^2(x^2-2\ep x+3);\label{trioour1}\\
& 2x^5-3x^3+27=(x^2-3x+3)(2x^3+6x^2+9x+6).\notag
\end{align}

Later we consider the case $p^k=a+b$ for $f(x)=ax^n+b\ep_1x^m+p^k\ep_2, m\le 3$ and proved that 

\begin{theorem}[Koley \& Reddy \cite{biswajit4}]\label{primepower}
Let $f(x)=ax^n+b\ep_1x^m+p^k\ep_2\in \S_1 \cup \S_2' \cup \S_3'$ be a polynomial of degree $n$. If $p^k=a+b$ and $f(x)$ is reducible, then  $f(x)=f_c(x)f_n(x),$ where $f_c(x)=(x^n+\ep_2, x^m+\ep_1\ep_2)$ is the product of all cyclotomic factors of $f$ and $f_n(x)$ is an irreducible non-reciprocal polynomial. In particular, $f(x)$ is irreducible except for the following families of polynomials:

\begin{tabular}{|l|l|}
\hline
$f(x) $ & $f_c(x)$\\
\hline
$ax^{2u+1}+bx+p^k\ep_2 $ & $x+\ep_2$\\[5pt]
\hline
$ax^{2u}-bx-p^k$ & $ x+1$\\[5pt]
\hline
$ax^{2(2u+1)}+bx^2+p^k\ep_2$ & $ x^2+\ep_2$\\[5pt]
\hline
$ax^{2u+1}-bx^2+p^k$ & $ x+1$\\[5pt]
\hline
\end{tabular}
\quad
\begin{tabular}{|l|l|}
\hline
$f(x)$ & $f_c(x)$\\
\hline
$ax^{2^{2t}(2u+1)}-bx^2-p^k$ & $x^2+1$\\[5pt]
\hline
$ax^{2(2u+1)}+bx^2+p^k\ep_2$  & $ x^2+\ep_2$\\[5pt]
\hline
$ax^{2u+1}+bx^3+p^k\ep_2$ & $x^{(2u+1,3)}+\ep_2$\\[5pt]
\hline
$ax^{2^t(2u+1)}-bx^3-p^k$ & $x^{(2u+1,3)}+1$\\[5pt]
\hline
\end{tabular}\\

for every $u, t\ge 1$. 
\end{theorem}

Example \eqref{trieq1} shows that Theorem \ref{primepower} is not true for $f(x)\in \S_2\setminus \S_2'$. The following example illustrate that it is neither true for $f(x)\in \S_3\setminus \S_3'$ as well
\begin{equation*}
3x^4-5x^3-8=(x+1)(x-2)(3x^2-2x+4).
\end{equation*}

\end{enumerate}

\subsection{Reducibility of $\boldsymbol{ax^n+b\ep_1x^m+c\ep_2}$}
Suppose $f(x)=ax^n+b\ep_1x^m+c\ep_2$ is a trinomial of degree $n$, where $a,b,c\in \N$. If $a, c>1$, then there are very few results available related to the reducibility criterion.  Theorem \ref{bonciocattrinomial} and Theorem  \ref{bonciocattrinomial1} gives an irreducibility criterion when the constant term is divisible by a large prime power. A. Schinzel \cite{AS} found the conditions on coefficients so that an arbitrary trinomial is divisible by a cyclotomic polynomial. 

\begin{theorem}[Schinzel \cite{AS}]\label{schinzeltrinomialcyclo}
Let $a,b,c$ are non-zero integers, $0<m<n, d=(m,n), n=dn_1, m=dm_1$. If $f_c(x)$ is the cyclotomic part of $f(x)=ax^n+bx^m+c$, then 
\footnotesize{
\begin{equation*}
f_c(x)=\begin{cases}
x^{2d}+\ep_1^{m_1+n_1}\ep_2^{m_1}x^d+1, &\mbox{ if $c=\ep_1a=\ep_2b, n_1+m_1\equiv 0\pmod{3}, \ep_1^{m_1}=\ep_2^{n_1}$;}\\
(x^d-(-\ep_1)^{m_1}\ep_1\ep_2)^2, &\mbox{ if $c=\ep_1a+\ep_2b, (-\ep_1)^{m_1}=(-\ep_2)^{n_1}, an\ep_1+bm\ep_2=0$;}\\
x^d-(-\ep_1)^{m_1}\ep_1\ep_2,  &\mbox{ if $c=\ep_1a+\ep_2b, (-\ep_1)^{m_1}=(-\ep_2)^{n_1}, an\ep_1+bm\ep_2\ne 0$;}\\
1 &\mbox{ otherwise.}
\end{cases}
\end{equation*}
}
\end{theorem}

Schinzel also proved that there are only finitely many trinomials $ax^n+bx^m+c$ whose non-cyclotomic part is reducible. More precisely, 

\begin{theorem}[Schinzel \cite{AS}]\label{schinzeltrinomial}
For any non-zero integers $a,b,c$ there exist two effectively computable constants $A(a,b,c)$ and $B(a,b,c)$ such that if $n>m>0$ and the non-cyclotomic part of $ax^n+b^m+c$ is reducible, then
\begin{enumerate}[label=(\alph*)]
\item $\frac{n}{(n,m)}\le A(a,b,c)$,
\item there exists integers $\alpha$ and $\beta$ such that $\frac{m}{\beta}=\frac{n}{\alpha}$ is integral, $0<\beta<\alpha\le B(a,b,c)$ and if $g(x)$ is an irreducible non-reciprocal factor of $ax^{\alpha}+bx^{\beta}+c$ then $g(x^{n/\alpha})$ is an irreducible non-reciprocal factor of $ax^n+bx^m+c$. 
\end{enumerate}
\end{theorem}

All the trinomials that we came across so far have at most two non-reciprocal irreducible factors. We recall the two sets $\S_m$ and $\S_m'$ introduced earlier. From Theorem \ref{bonciocattrinomial1}, if $f(x)\in \S_2'$ and $p^k>a+b$, then $f(x)$ is irreducible. We have shown in \cite{biswajit4} that if $f(x)\in \S_2\setminus \S_2'$ and $p^k>a+b$, then $f(x)$ can have at most two non-reciprocal irreducible factors.  Apart from the polynomials given in Theorem \ref{jonassen}, many polynomials having two irreducible non-reciprocal factors and $k$ even. For example, if $p$ is a prime and $k\ge 1, a\ge 1$, then
\begin{align}
& x^4+(2\cdot p^k-a^2)x^2+p^{2k}=(x^2+ax+p^k)(x^2-ax+p^k)\in \S_2\setminus\S_2'\notag
\end{align}
is a product of two irreducible non-reciprocal factors. Later we extend it and proved that if $f(x)\in \S_3\setminus \S_3'$ and $p^k>a+b,$ then $f(x)$ can have at most three non-reciprocal irreducible factor.  The sharpness of this result  follows from example \eqref{trioour1}. 
However, the number of  non-reciprocal  irreducible factors of $ax^n+bx^m+c$ is unbounded. For example, 
\begin{equation}\label{arbitraryfactor}
x^{2k}-2^{k+1} x^k+2^{2k}=(x^k-2^k)^2=\left(\prod\limits_{d|k}2^{\varphi(d)}\Phi_d\left(\frac{x}{2}\right) \right)^2,
\end{equation}
where $2^{\varphi(d)}\Phi_d\left(\frac{x}{2}\right)\in \Z[x]$ is an irreducible non-reciprocal polynomial for every $d\in \N$.
Over time, several other questions were asked related to  the irreducibility of trinomials. For example, Schinzel \cite{AS5} asked: 
\begin{qu}
Does there exist an integer $K,$ independent of the polynomial, such that every trinomial in $\Q[x]$ has an irreducible factor in $\Q[x]$ with at most $K$ terms?
\end{qu}

Mrs. H. Smyczek \cite{AS5} showed that if $K$ exists, then $K\ge 6$. Bremner \cite{bremner1} improved to $K\ge 8$ by furnishing an exact trinomial of the form $x^{14}+bx^2+c,$ having precisely two factors each of length $8$.  However, the question remains open as of today.

Inspired by the work of Bremner \cite{bremner1}, several authors tried to characterize trinomials as having a factor of fixed length. For further studies in this direction, one can look into \cite{chen}, \cite{herendi},  \cite{le},  \cite{lin},\cite{liu}, \cite{rabinowitz},  \cite{ribenboim}, \cite{AS8},\cite{AS9}, \cite{AS10}, \cite{yangfu}.

Let $0<n_1\le n_2\le \cdots\le n_k$. A trinomial $x^n+bx^m+c\in \Q[x]$ is said to have reducibility type $(n_1,n_2,\ldots, n_k),$  if it is a product of irreducible polynomials in $\Q[x]$ of degree $n_1, n_2,\ldots, n_k$. For example, if we take $k=15$ in \eqref{arbitraryfactor}, then
\begin{small}
\begin{equation*}
x^{30}-2^{16}x^{15}+2^{30}=(x-2)^2(x^2+2x+4)^2(x^4+2x^3+4x^2+8x+16)^2(x^8-2x^7+\cdots+32x^3-128x+256)^2
\end{equation*}
\end{small}
has a reducibility type $(1,1,2,2,4,4,8,8)$. A. Bremner and M. Ulas \cite{bremner2} studied possible reducibility type  for $n\le 10$ and $m\le 3$. They showed that if $m=1,$ then the reducibility type $(1,1,1,1)$ is impossible.  It has been conjectured that

\begin{conj}[Bremner \& Ulas \cite{bremner2}]\label{bremnerconjecture}
Let $b, c\in \N$ and $f(x)=x^n+b\ep_1x+c\ep_2$ be a trinomial of degree $n\ge 4$. Then the reducibility type $(1,1,1,n-3)$ doesn't exist. 
\end{conj} 

Suppose $f(x)\in \Q[x]$ is a polynomial which divides infinitely many trinomials. Can $f(x)$ divide a linear or quadratic polynomial in $x^r$ for some $r\ge 1$? For example, $\Phi_3(x)$ divides $x^{3k+2}+x+1$ for every $k\ge 1$, by Theorem \ref{selmer1}, and $\Phi_3(x)|\Phi_3(x^r)=x^{2r}+x^r+1$ for every $r\ge 1, (r,3)=1$.  Is this true for arbitrary $f(x)\in \Q[x]?$ Posner and Rumsey \cite{posner} investigated this question and subsequent development can be found in \cite{gyory}, \cite{viola}. 

Till now we were confined to mainly questions regarding trinomials over field of rational numbers. Irreducibility of trinomials over arbitrary field is not that much common in the literature. Let $\K$ be an algebraic number field and $f(x)=x^n+bx^m+c\in \K[x]$ has a linear or quadratic factor $g(x)$. Schinzel \cite{AS13} studied reducibility of the part $h(x)=\frac{f(x)}{g(x^{(n,m)})}$ and gave several necessary and sufficient condition for reducibility of $h(x)$ over  $\K$. 
For irreducibility of trinomials over the function field, we refer to \cite{AS9}, \cite{ AS10}, \cite{ AS11}. 

The reducibility of trinomials over finite field has also been studied extensively. There is a classical result, known as modulo $p$-test ($p$ being a prime number), which relates the irreducibility of polynomials over a finite field and the irreducibility of polynomials over $\Q$.\\
{\bf Modulo $p$-test(\cite{gallian}, Theorem 17.3): } Let $p$ be a prime and suppose that $f(x)\in \Z[x]$ with $\deg(f(x))\ge 1$. Let $\tilde{f}(x)$ be the polynomial in $\Z_p[x]$ obtained from $f(x)$ by reducing all the coefficients of $f(x)$ modulo $p$. If $\tilde{f}(x)$ is irreducible over $\Z_p$ and $\deg(\tilde{f}(x))=\deg(f(x))$, then $f(x)$ is irreducible over $\Q$. \\

For a general overview of polynomial factorization over finite field, the reader can consider \cite{ore1}, \cite{vandiver1}.
Swan \cite{swan} proved that if $\F$ is a field of odd characteristic and $f(x)\in \F[x]$ is a polynomial of degree $n$, having no repeated roots, then $r\equiv n\pmod{2}$ if and only if discriminant (see section \ref{sec:dt} for definition) of $f(x)$ is a square in $\F$, where $r$ is the number of irreducible factors of $f(x)$. Dickson \cite{dickson} also proved this result independently. L. Carlitz \cite{carlitz} considered the trinomial $g(x)=x^{2q^n+1}+x^{q^n-1}+1$ and proved that the degree of every irreducible factor of $g(x)$ over $\F_q$ divides either $2n$ or $3n$, where $q$ is a prime power.

The distribution of irreducible trinomials over $\F_3$ has been studied by O. Ahmadi \cite{ahmadi}. It has been proved that if $n$ is even, $q$ is an odd prime power, then $x^n+bx^m+c$ is irreducible over $\F_q$ if and only if $x^n-bx^m+c$ is irreducible over $\F_q$. Later Ahmadi proved that if $4|n,$ then $x^n-x^m+1$ is irreducible over $\F_3$. A necessary condition for the irreducibility of trinomials over a finite field is given by  J. Gathen \cite{von1}:

\begin{theorem}[Gathen \cite{von1}]
Let $p$ be an odd prime, $f(x)=x^n+bx^m+c, b,c\in \F_{p^r}\backslash \{\, 0\,\}, r\ge 1, n>m\ge 1,d=(n,m), n_1=n/d, m_1=m/d, k_1=p(p^r-1), k_2=\lcm(4, k_1).$ Then the discriminant of $f$ and the property that $f$ is squarefree with odd number of irreducible factors depends only on the following residues:

$$n\pmod{k_2}; \quad m\pmod{k_1}; \quad n_1\pmod{p^r-1}\quad  \text{ and }\quad m_1\pmod{p^r-1}.$$
\end{theorem}

For further study on the factorization of trinomials over finite field, reader should refer to \cite{agou},  \cite{albert}, \cite{blake},  \cite{estes}, \cite{fredricksen},  \cite{vishne},  \cite{von}.

\section{Location of zeros of trinomials}\label{sec:lzt}
The location of zeros of a polynomial sometimes helps us in  determining the irreducibility of the polynomial.  For example, if a monic polynomial $f(x)\in \Z[x], f(0)\ne 0$ has exactly one root in $|z|>1$ and all the remaining roots are in $|z|<1$, then $f(x)$ is irreducible. For if $f(x)=f_1(x)f_2(x)$ in $\Z[x]$, then either $f_1(x)$ or $f_2(x)$ has all its root in $|z|<1$, contradicting the fact $f_1(x), f_2(x)\in \Z[x]$. Similarly, if a monic polynomial $f(x)\in \Z[x], f(0)\ne 0$ has $k$ roots in $|z|>1$ and remaining roots are in $|z|<1$, then $f(x)$ can have at most $k$ irreducible factors. In this section, we  will discuss about the location of zeros of trinomials. Unless specified otherwise, throughout the section.  it is  assumed that all the coefficients are non-zero real numbers. Kennedy \cite{kennedy} specified the region in which all the roots of the trinomial $x^n+bx^m+c$ lie. 
\begin{theorem}[Kennedy \cite{kennedy}]
If $\alpha$ is the positive real root of the equation $x^n+|b|x^m-|c|=0$ and $\beta$ is the positive real root of $x^n-|b|x^m-|c|=0$, then the roots of the trinomials $x^n+bx^m+c$ lies in the annulus $\alpha\le |z|\le \beta$. 
\end{theorem}

However, finding a positive real root of $x^n+b x^m-|c|$ is not an easy task to accomplish. Because of this reason, Kennedy gave a weaker but easily computable bound. 
\begin{theorem}[Kennedy \cite{kennedy}]
If $\alpha$ and $\beta$ are the positive real roots of the equation $x^{n-m}=|c|^{n-m/n}-|b|$ and $x^{n-m}=|c|^{n-m/n}+|b|$ respectively, then the roots of $x^n+bx^m+c$ lie in the annulus $\alpha<|z|<\beta$. 
\end{theorem}

Suppose $a,b,c$ are non-zero real numbers, $|a|+|c|<|b|$, and  $f(x)=ax^n+bx^m+c$. For $|z|=1$, 
\begin{equation*}
|az^n+bz^m+c-bz^m|=|az^n+c|\le |a|+|c|<|b|=|bz^m|.
\end{equation*}
From Rouch\'e's theorem, $f(x)=ax^n+bx^m+c,$ $|a|+|c|<|b|$ has $m$ number zeros in $|z|<1$. If $|a|+|c|=|b|$, then the location of zeros has been calculated by Dilcher et al. \cite{dilcher}. They considered trinomials of the form $bx^n-ax^m+a-b$ with $a>b>0$ and  proved that 
\begin{theorem}[Dilcher et al. \cite{dilcher}]\label{dilcher}
Suppose $a>b>0$ are real numbers. The number of zeros of $f(x)=bx^n-ax^m+a-b$ strictly inside the unit circle is $m-\gcd(n,m)$ if $\frac{a}{b}\ge \frac{n}{m}$, and $m$ if $\frac{a}{b}<\frac{n}{m}$.  
\end{theorem}

M. A. Brilleslyper and L.E. Schaubroeck \cite{brilleslyper} counted the number of roots in the interior of the unit circle for the trinomial $x^n+x^m-1$. They proved that if $m=1$, then there are exactly $2\lfloor{\frac{n}{6}}\rfloor+1 $ roots in the interior of the unit circle, where $\lfloor{x }\rfloor$ denotes the greatest integer less than or equal to $x$. They further conjectured that 
\begin{conj}[Brilleslyper \& Schaubroeck \cite{brilleslyper}]\label{conjectureroot}
Let $n,m\in \N$ and $(n,m)=1$. Then the number of roots in the interior of the unit circle for the polynomial $x^n+x^m-1$ is 
\begin{equation*}
2\Bigl\lfloor{\frac{n+m-1}{6}}\Bigr\rfloor+1.
\end{equation*}
\end{conj}

The remaining results  available in the literature focus on  the zeros of trinomials of the form $ax^n+bx+c$. For example, Nicolas and Schinzel \cite{AS12} studied the location of zeros of  $x^{n+1}-ax+(a-1),$  and later the results have been improved further by Hernane and Nicolas \cite{hernane}. The zeros of $(a-2)x^n+(a-1)x-b$, $a>2$ has been studied in \cite{ahn}. Fej\'er \cite{fejer} and Szeg$\hat{o}$ \cite{szego} independently showed that the trinomial $cx^n-x+1$ has root both in the regions $|z-1|\ge 1$ and $|z-1|\le 1$, for any non-zero real number $c$. Joyal et.al.  \cite{joyal} proved that if $n\ge 3,$ then the polynomial $cx^n-x+1$ has a root outside every circle that passes through the origin. 

P. Vassilev \cite{peter} considered trinomials with complex coefficients. He proved that the trinomial $z^n-pz-1$ has only one root in the interval $(0,1)$ when $p<0$. Recently A. Melman \cite{aaron} studied the problem of the location of zeros of trinomials extensively and improved many existing known results. 

The number of roots of a trinomial over a finite field has been estimated in \cite{kelley}. For similar studies over finite fields, one may look into   \cite{robert}, \cite{hua}, \cite{vandiver}, \cite{vandiver2}.

\section{Discriminant of trinomials}\label{sec:dt}
Suppose $f(x)=a_nx^n+a_{n-1}x^{n-1}+\cdots+a_1x+a_0\in \Z[x]$ is a polynomial of degree $n$. Let $\alpha_1, \alpha_2, \ldots, \alpha_n$ be the roots of $f(x)$. The {\em discriminant}  of $f(x)$ is defined as 
\begin{equation*}
D_f=a_n^{2n-2}\prod\limits_{i<j}(\alpha_i-\alpha_j)^2=(-1)^{\binom{n}{2}}a_n^{2n-2}\prod\limits_{i\ne j} (\alpha_i-\alpha_j).
\end{equation*}
The quantity $(-1)^{\binom{n}{2}}a_nD_f$ is called the {\em resultant} of $f(x)$. More precisely, the resultant of $f(x)$, sometimes denoted by $\Res(f, f')$, is defined as 
\begin{equation*}
\Res(f, f') =(-1)^{\binom{n}{2}}a_n^{2n-1}\prod\limits_{i<j}(\alpha_i-\alpha_j)^2=a_n^{2n-1}\prod\limits_{i\ne j}(\alpha_i-\alpha_j),
\end{equation*}
where $f'(x)=na_nx^{n-1}+(n-1)a_{n-1}x^{n-2}+\cdots+a_1$ is the formal derivative of $f(x)$. 

From the definition, it follows that $f(x)\in \Z[x]$ has  multiple roots if and only if $D_f=0$. If $f(x)\in \Z[x]$ has  multiple roots, then $f(x)$ is said to be a {\em non-separable} polynomial over $\Q$. Otherwise, it is said to be a {\em separable} polynomial  over $\Q$. 

There is another concept, called the discriminant of an algebraic number field.  Suppose $\K$ is an algebraic number field and $\Od_{\K}$ is the ring of integers of $\K$. Let $b_1, b_2,\ldots, b_n$ be an integral basis of $\Od_{\K}$ and let $\{\, \sigma_1, \ldots, \sigma_n\,\}$ be the set of embeddings of $\K$ into $\C$. Then the {\em discriminant } of $\K$, denoted by $\Delta_{\K}$, is defined as  
\begin{equation*}
\Delta_{\K}=\det(B)^2, \,\,\qquad \mbox{ where $B=(\sigma_i(b_j))_{i,j}$.}
\end{equation*}

\begin{ex*}
Let $d$ be a square-free integer and let $\K=\Q(\sqrt{d})$ be a quadratic field.  Then the discriminant of $\K$ is  (see \cite{nart} for details) 
\begin{equation*}
\Delta_{\K}=\begin{cases}
d, &\mbox{ if $d\equiv 1\pmod{4};$}\\
4d, &\mbox{ if $d\equiv 2, 3\pmod{4}.$}
\end{cases}
\end{equation*}
\end{ex*}
The discriminant of a polynomial and discriminant of an algebraic number field are closely related. Let $f(x)$ be a monic irreducible polynomial with a root $\alpha$ and $\K=\Q(\alpha)$. Then the discriminant of $f(x)$ is 
\begin{equation*}
D_f=i(\alpha)^2 \Delta_{\K}, \,\, \qquad \mbox{ where $i(\alpha)=\# [\Od_{\K}: \Z[\alpha]].$}
\end{equation*}

If $D_f$ is square-free, then $i(\alpha)^2=1$ would imply that $\Od_{\K}=\Z[\alpha],$ and hence $\Od_{\K}$ has a basis in powers of $\alpha$. In this case, $\K$ is said to be a {\em monogenic field}. There are several advantages of knowing a field to be monogenic. For example, let $\K$ be a monogenic field. Then the Galois closure $\L$ of $\K$ has Galois group $S_n$ and $\L$ is everywhere unramified over its subfield $\Q(\sqrt{D_f})$.  Therefore, the knowledge of the square-freeness of the discriminant of a polynomial is useful. However, it is very difficult to provide a formula for the discriminant of a polynomial with more than three non-zero terms. To the best of our knowledge, 
only quadrinomials for which the discriminant is known as of today are of the form $x^n+t(x^2+bx+c)$  (see \cite{otake} for details).

The reader can refer to Chapter 10 of \cite{emilartin} for a detailed literature review on the discriminants of trinomials up to 1980.  P. Lefton \cite{lefton} proved the discriminant formula of an arbitrary trinomial by using a polynomial and its derivatives along with the properties of roots of unity. He showed that the discriminant of $ax^n+bx^m+c$ is 
\begin{equation*}
D=(-1)^{\binom{n}{2}}a^{n-m-1}c^{m-1} (n^{n/d}a^{m/d}c^{(n-m)/d}-(-1)^{n/d}(n-m)^{(n-m)/d}m^{m/d}b^{n/d})^d,
\end{equation*}
where $d=(n,m)$. D. Drucker and D. Goldschmidt \cite{drucker} developed an ingenious method to compute the determinant of a matrix. By using this method,  G.R.Greenfield and D. Drucker \cite{greenfield} proved that
 
\begin{theorem}[Greenfield \& Drucker \cite{greenfield}]\label{greenfield}
The discriminant of the trinomial $x^n+bx^m+c$ is 
\begin{equation*}
D=(-1)^{\binom{n}{2}}c^{m-1}\left[ n^{n/d}c^{n-m/d}-(-1)^{n/d}(n-m)^{n-m/d}m^{m/d}b^{n/d}\right]^d,
\end{equation*}
where $d=(n,m)$.
\end{theorem}
Let $(n,m)=1$ and $f(x)=x^n\pm x^m\pm 1$. From Theorem \ref{greenfield}, the discriminant of $f$ is  
$$D_f=\pm n^n\pm (n-m)^{n-m}m^m.$$ 

Recently D. Boyd et al. \cite{boyd1} attempted to find the square-free values of this discriminant. They found several families of polynomials whose discriminant has a square factor. For example, if $n\equiv 2\pmod{6},$ then $\frac{(n^2-n+1)^2}{9}$ divides $n^n-(n-1)^{n-1}$. That is, the discriminant of  $x^n-x-1$ has a square-factor when $n\equiv 2\pmod{6}$. However, such a discriminant is sporadic. If $n\le 1000$, then there are only $6$ such values of $n$, $n\in \{\, 130, 257, 487, 528,815, 897\,\}$. It can be seen that $83^2|130^{130}+129^{129}$ and $59^2$ divide  $n^n+(-1)^n(n-1)^{n-1}$ for the remaining values of $n$. They conjectured that 
\begin{conj}[Boyd et al. \cite{boyd1}]
The set of positive integers $n$ such that $n^n+(-1)^n(n-1)^{n-1}$ is square-free has density $0.9934466\ldots $, correct to that many decimal places. 
\end{conj}   

They obtain $0.99344674$ as an upper bound for this density. I.E. Shparlinski \cite{igor} gave a lower bound of the above density. 

If the $abc$ conjecture is true, then Anirban Mukhopadhyay et al. \cite{anirban} showed that for every odd $n\ge 5$, there is a positive proportion of pairs $(b,c)$ for which the trinomial $x^n+bx+c$ is irreducible and has square-free discriminant. By using square-sieve and bounds of character sums, I.E. Shparlinski \cite{igor1} proved a weaker but unconditional version of this result. K. Kedlaya\cite{kedlaya} gave a method to construct a monic irreducible polynomial $f(x)\in \Z[x]$ of degree $n\ge 2$ having exactly $r$ real roots, $0\le r\le n,$ with a square-free discriminant.  

On the other hand, discriminants of algebraic number fields have also been explored by several mathematicians. Suppose $\alpha$ is a root of the irreducible polynomial $x^n+bx+c$ and $\K=\Q(\alpha)$. K. Komatsu \cite{komatsu} found the discriminant of $\K$ as well as an integral basis for $\K$. Later, P. Llorente et al. \cite{vila}  improved the results of Komatsu.  

B. K. Spearman and K.S. Williams \cite{spearman} restricted to trinomials $f(x)=x^5+ax+b,$ whose Galois group is a dihedral group of order $10$. If $\alpha$ is a root of $f$ and $\K=\Q(\alpha)$, then they give an explicit formula of $\Delta_{\K}$ in terms of $a$ and $b$.  
A. Alaca and \c{S}. Alaca \cite{alaca} utilized a method described by \c{S}. Alaca \cite{alaca1} and found a $p$-integral base of $\K$ for every prime $p$, which in turn helps them to provide a formula for $\Delta_{\K}$, where $\K=\Q(\alpha), \alpha$ being a root of the irreducible polynomial $x^5+bx+c$.

We have seen that the polynomial $x^n-x-1$ is irreducible with discriminant $n^n+(-1)^n(n-1)^{n-1}$ and has Galois group $S_n$. Boyd et al. \cite{boyd1} showed that there are certain values of $n$ (see above) for which the discriminant of $x^n-x-1$ has a square factor. J. Lagarias(see problem 99:10, \cite{lagarias}) asked 

{\em   Is there a monogenic field whose Galois group is $S_n$ for every $n\ge 5$? }

K. Kedlaya \cite{kedlaya} answered this question affirmatively through a different approach. Boyd et al. \cite{boyd1} answered this by using square freeness of discriminant. For similar studies in this direction, we refer to  \cite{uchida}, \cite{uchida1}, \cite{yamamoto}, \cite{yamamura}.

Galois group of trinomials of different forms has been considered by several authors. One may consult \cite{bishnoi}, \cite{stephen}, \cite{brown}, \cite{bruin}, \cite{cohen}, \cite{cohen1}, \cite{hermez},  \cite{jones}, \cite{osada}, \cite{osada1}, \cite{spearman1} for further studies. For example, H. Osada \cite{osada} proved that the Galois group of $x^n-x-1$ over $\Q$ is $S_n$ for all $n\ge 2$. Later Osada \cite{osada1} gave some number-theoretic conditions on the coefficients $b,c$ so that $x^n+bx^m+c$ has Galois group either $S_n$ or $A_n$ over $\Q$.

In contrast, a class of trinomials has been determined by A. Bishnoi and S.A. Khanduja  \cite{bishnoi} having Galois group $S_n$. In particular, they showed that if $n\ge 8, \frac{n}{2}<p<n-2$ and  $|c^{n-p}+(-1)^{n+1}(\frac{n-p}{p})^{n-p}(n!)^pb^n|$ is not a perfect square, then $x^n+n!p^nbx^p+n!c$ has Galois group $S_n,$ $p$ being a prime number. One such polynomial is $x^6+864x^5+720$.

The discriminant of trinomials has been explored over finite fields as well.  The discriminant of a trinomial divisible by a square over a finite field has been entreated in \cite{hanson}. Suppose $\F$ is a finite field of order $q=p^n$ for some prime $p$. Then a polynomial $f(x)\in \F[x]$ is said to be a {\em permutation polynomial} if $c\mapsto f(c)$ is a permutation from $\F$ to itself. Permutation trinomials has been studied in \cite{bartoli}, \cite{ding}, \cite{hou}, \cite{hou1}, \cite{hou2}.

\section{Open Problems}
We have already mentioned several conjectures, open problems, and asked a few questions in respective places. In this section, we would like to ask few more questions related to the irreducibility of trinomials and locations of zeros of trinomials which may give impetus to the reader to work further in this direction. 

Recall that the reducibility nature of $x^n+2\ep_1x^m+\ep_2$ has been resolved by Theorem \ref{filasetatrinomial}. If $p$ is an odd prime, from Theorem \ref{minkusinski}, there are only finitely many values of $\frac{n}{m}$ for which $x^n+p\ep_1x^m+\ep_2$ is reducible.
\begin{qu}
Let $p$ be an odd prime, $\ep_1,\ep_2\in \{\, -1, +1\,\}$ and $n>m\ge 1$. What are the values of $n$ and $m$ for which $x^n+p\ep_1x^m+\ep_2$ is reducible?
\end{qu}
From Theorem \ref{schinzeltrinomial}, there are only finitely many trinomials $x^n+b\ep_1x^m+c\ep_2$ whose non-cyclotomic part is reducible. However, it is not known whether there are only finitely many reducible trinomials of the form $x^n+b\ep_1x^m+c\ep_2$. In particular, 
\begin{qu}
Let $n>m\ge 1, b\in \N,$ $\ep_1,\ep_2\in \{\, -1, +1\,\}$. Are there finitely many reducible trinomials of the form $x^n+b\ep_1x^m+\ep_2$?
\end{qu}

From Theorem \ref{ljunggrentrinomial} and Theorem \ref{biswajit2trinomial}, the trinomial $x^n+\ep_1x^{n-1}+\ep_2$ is reducible only when $n=6t+2, \ep_1=-\ep_2=-1$ or $n=6t+5, \ep_1=\ep_2=-1,$ $t\in \N$. 
If $b\ge 2$, then the reducibility of $x^n+b\ep_1x^{n-1}+c\ep_2, b\ne c$ has been determined under the certain condition on $c$ in Theorem \ref{harringtonsummary}. For $b=c\ge 2$, Harrington conjectured (recall Conjecture \ref{harringtonconjecture}) that apart from two polynomials, $x^n+b\ep_1x^{n-1}+b\ep_2$ is irreducible.
\begin{qu}
Let $c\in \N$ and $c>1$, $\ep_1,\ep_2\in \{\, -1, +1\,\}$. What are the values of $c$ for which $x^n+\ep_1x^{n-1}+c\ep_2$is irreducible?
\end{qu}
More generally, we ask the following question
\begin{qu}
Let $b,c\in \N, b, c\ge 2,$ and $c>2(b-1),$ $\ep_1,\ep_2\in \{\, -1, +1\,\}$. For what values of $b$ and $c$, the polynomial $x^n+b\ep_1x^{n-1}+c\ep_2$ is irreducible?
\end{qu}

Filaseta et al. \cite{filaseta8} gave a sufficient condition for the irreducibility of $x^{2p}+b\ep_1x^p+c\ep_2$. There are only finitely many values of $p$ for which the irreducibility question remains open. 

\begin{qu}
Let $p$ be a prime number, $b,c\in \N$, $\ep_1,\ep_2\in \{\, -1, +1\,\}$ and $$p|(181^2-1)\underset{q|b}{\prod\limits_{q \text{prime}}} (q^2-1).$$ If $b, c$ do not satisfy any of the following conditions
\begin{equation*}
    b^2-4c\ep_2=d^2 \text{ or } c=b^2=d^{2p}, p\ge5 \text{ or } c=\frac{b^2}{2}=2^pd^{2p}, p\ge 3 \text{ or } c=\frac{b^2}{3}=3^pd^{2p}, p\ge 5
\end{equation*}
for some $d\in \Z,$ then which are the families among $x^{2p}+b\ep_1x^p+c\ep_2$ irreducible?
\end{qu}
In Theorem \ref{constantprimesummary}, the irreducibility of $x^n+b\ep_1x^m+p\ep_2$ has been discussed except for $p+1\le b\le p^{n-1}+1, p\nmid b$. 

\begin{qu}
Let $p$ be a prime number, $b\in \N$, and let $f(x)=x^n+b\ep_1x^m+p\ep_2.$ For what values of $b$, $p+1\le b\le p^{n-1}+1, p\nmid b$, the polynomial $f(x)$ is irreducible?
\end{qu}

The number of roots of $x^n+x-1$ which lies inside the unit circle has been counted in \cite{brilleslyper} and they conjectured (see Conjecture \ref{conjectureroot}) the exact number of roots within the unit circle for arbitrary values of $m$.

\begin{qu}
Let $n>m\ge 1$ and $\ep_1,\ep_2\in \{\, -1, +1\,\}$. How many roots of $x^n+\ep_1x^m+\ep_2$ lies in the interior of the unit circle?
\end{qu}

More generally, one can look for polynomials of the form $f(x)=ax^n+bx^m+c,$ where $a,b,c$ are non-zero real numbers. From Theorem \ref{dilcher} and related discussion, one can find the number of roots of $f(x)$ that are inside the unit circle provided $|b|\ge |a|+|c|$. 

\begin{qu}
Let $a,b,c\in \R\backslash\{\, 0\,\}$ and $n>m\ge 1.$ How many roots of $ax^n+bx^m+c$ lies in the interior of the unit circle?
\end{qu}

{\bf Acknowledgement:} We thank the reviewer for providing valuable suggestions and comments on an earlier version of the manuscript.

\bibliographystyle{plain}
\bibliography{refer}

\begin{thebibliography}{100}

\bibitem{agou}
S.~Agou.
\newblock Sur l'irr\'{e}ductibilit\'{e} des trin\^{o}mes {$X^{p^r+1}-aX-b$} sur
  les corps finis {${\bf F}_{p^s}$}.
\newblock {\em Acta Arith.}, 44(4):343--355, 1984.

\bibitem{ahmadi}
Omran Ahmadi.
\newblock On the distribution of irreducible trinomials over {$\Bbb F_3$}.
\newblock {\em Finite Fields Appl.}, 13(3):659--664, 2007.

\bibitem{ahn}
Young~Joon Ahn and Seon-Hong Kim.
\newblock Zeros of certain trinomial equations.
\newblock {\em Math. Inequal. Appl.}, 9(2):225--232, 2006.

\bibitem{alaca}
Ay\c{s}e Alaca and \c{S}aban Alaca.
\newblock An integral basis and the discriminant of a quintic field defined by
  a trinomial {$x^5+ax+b$}.
\newblock {\em JP J. Algebra Number Theory Appl.}, 4(2):261--299, 2004.

\bibitem{alaca1}
\c{S}aban Alaca.
\newblock {$p$}-integral bases of algebraic number fields.
\newblock {\em Util. Math.}, 56:97--106, 1999.

\bibitem{albert}
A.~A. Albert.
\newblock On certain trinomial equations in finite fields.
\newblock {\em Ann. of Math. (2)}, 66:170--178, 1957.

\bibitem{emilartin}
Emil Artin.
\newblock {\em Theory of algebraic numbers}, volume 1956/7 of {\em Notes by
  Gerhard W\"{u}rges from lectures held at the Mathematisches Institut,
  G\"{o}ttingen, Germany, in the Winter Semester}.
\newblock George Striker, Schildweg 12, G\"{o}ttingen, 1959.

\bibitem{bartoli}
Daniele Bartoli and Giovanni Zini.
\newblock On permutation trinomials of type {$x^{2p^s+r}+x^{p^s+r}+\lambda
  x^r$}.
\newblock {\em Finite Fields Appl.}, 49:126--131, 2018.

\bibitem{bishnoi}
Anuj Bishnoi and Sudesh~K. Khanduja.
\newblock A class of trinomials with {G}alois group {$S_n$}.
\newblock {\em Algebra Colloq.}, 19(Special Issue 1):905--911, 2012.

\bibitem{blake}
Ian~F. Blake, Shuhong Gao, and Robert~J. Lambert.
\newblock Construction and distribution problems for irreducible trinomials
  over finite fields.
\newblock In {\em Applications of finite fields ({E}gham, 1994)}, volume~59 of
  {\em Inst. Math. Appl. Conf. Ser. New Ser.}, pages 19--32. Oxford Univ.
  Press, New York, 1996.

\bibitem{bonciocat1}
Anca~Iuliana Bonciocat and Nicolae~Ciprian Bonciocat.
\newblock Some classes of irreducible polynomials.
\newblock {\em Acta Arith.}, 123(4):349--360, 2006.

\bibitem{borweinbook}
Peter Borwein and Tam\'{a}s Erd\'{e}lyi.
\newblock {\em Polynomials and polynomial inequalities}, volume 161 of {\em
  Graduate Texts in Mathematics}.
\newblock Springer-Verlag, New York, 1995.

\bibitem{boyd}
David~W. Boyd.
\newblock The maximal modulus of an algebraic integer.
\newblock {\em Math. Comp.}, 45(171):243--249, S17--S20, 1985.

\bibitem{boyd1}
David~W. Boyd, Greg Martin, and Mark Thom.
\newblock Squarefree values of trinomial discriminants.
\newblock {\em LMS J. Comput. Math.}, 18(1):148--169, 2015.

\bibitem{bremner1}
Andrew Bremner.
\newblock On reducibility of trinomials.
\newblock {\em Glasgow Math. J.}, 22(2):155--156, 1981.

\bibitem{bremner}
Andrew Bremner.
\newblock On trinomials of type {$x^{n}+Ax^{m}+1$}.
\newblock {\em Math. Scand.}, 49(2):145--155 (1982), 1981.

\bibitem{bremner2}
Andrew Bremner and Maciej Ulas.
\newblock On the reducibility type of trinomials.
\newblock {\em Acta Arith.}, 153(4):349--372, 2012.

\bibitem{brilleslyper}
Michael~A. Brilleslyper and Lisbeth~E. Schaubroeck.
\newblock Counting interior roots of trinomials.
\newblock {\em Math. Mag.}, 91(2):142--150, 2018.

\bibitem{stephen}
Stephen~C. Brown, Blair~K. Spearman, and Qiduan Yang.
\newblock On the {G}alois groups of sextic trinomials.
\newblock {\em JP J. Algebra Number Theory Appl.}, 18(1):67--77, 2010.

\bibitem{brown}
Stephen~C. Brown, Blair~K. Spearman, and Qiduan Yang.
\newblock On sextic trinomials with {G}alois group {$C_6$}, {$S_3$} or
  {$C_3\times S_3$}.
\newblock {\em J. Algebra Appl.}, 12(1):1250128, 9, 2013.

\bibitem{bruin}
Nils Bruin and Noam~D. Elkies.
\newblock Trinomials {$ax^7+bx+c$} and {$ax^8+bx+c$} with {G}alois groups of
  order 168 and {$8\cdot168$}.
\newblock In {\em Algorithmic number theory ({S}ydney, 2002)}, volume 2369 of
  {\em Lecture Notes in Comput. Sci.}, pages 172--188. Springer, Berlin, 2002.

\bibitem{capelli}
A.~Capelli.
\newblock Sulla riduttibilita delle equazioni algebriche.
\newblock {\em Nota prima, Red. Accad. Fis. Mat. Soc. Napoli(3)}, 3:243--252,
  1897.

\bibitem{carlitz}
L.~Carlitz.
\newblock Factorization of a special polynomial over a finite field.
\newblock {\em Pacific J. Math.}, 32:603--614, 1970.

\bibitem{chen}
Hong~Ji Chen.
\newblock On the quadratic factorization of {$x^n-x-a$}.
\newblock {\em J. Math. (Wuhan)}, 22(3):319--322, 2002.

\bibitem{cohen}
S.~D. Cohen, A.~Movahhedi, and A.~Salinier.
\newblock Galois groups of trinomials.
\newblock {\em J. Algebra}, 222(2):561--573, 1999.

\bibitem{cohen1}
Stephen~D. Cohen.
\newblock Galois groups of trinomials.
\newblock {\em Acta Arith.}, 54(1):43--49, 1989.

\bibitem{robert}
Robert Coulter and Marie Henderson.
\newblock A note on the roots of trinomials over a finite field.
\newblock {\em Bull. Austral. Math. Soc.}, 69(3):429--432, 2004.

\bibitem{dickson}
L.~E. Dickson.
\newblock Criteria for the irreducibility of functions in a finite field.
\newblock {\em Bull. Amer. Math. Soc.}, 13(1):1--8, 1906.

\bibitem{dilcher}
Karl Dilcher, James~D. Nulton, and Kenneth~B. Stolarsky.
\newblock The zeros of a certain family of trinomials.
\newblock {\em Glasgow Math. J.}, 34(1):55--74, 1992.

\bibitem{ding}
Cunsheng Ding, Longjiang Qu, Qiang Wang, Jin Yuan, and Pingzhi Yuan.
\newblock Permutation trinomials over finite fields with even characteristic.
\newblock {\em SIAM J. Discrete Math.}, 29(1):79--92, 2015.

\bibitem{drucker}
Daniel Drucker and David~M. Goldschmidt.
\newblock Graphical evaluation of sparse determinants.
\newblock {\em Proc. Amer. Math. Soc.}, 77(1):35--39, 1979.

\bibitem{smyth}
A.~Dubickas and C.~J. Smyth.
\newblock On the {R}emak height, the {M}ahler measure and conjugate sets of
  algebraic numbers lying on two circles.
\newblock {\em Proc. Edinb. Math. Soc. (2)}, 44(1):1--17, 2001.

\bibitem{dubickas}
Art\={u}ras Dubickas.
\newblock Nonreciprocal algebraic numbers of small measure.
\newblock {\em Comment. Math. Univ. Carolin.}, 45(4):693--697, 2004.

\bibitem{eisen}
F.G.M. Eisenstein.
\newblock Uber die irreducibilitat und einige andere eigenschaften der
  gleichung, von welcher die theilung der ganzen lemniscate abhangt.
\newblock {\em J. reine angew. Math.}, 39:166--167, 1850.

\bibitem{estes}
Dennis~R. Estes and Tetsuro Kojima.
\newblock Irreducible quadratic factors of {$x^{(q^n+1)/2}+ax+b$} over {${\bf
  F}_q$}.
\newblock {\em Finite Fields Appl.}, 2(2):204--213, 1996.

\bibitem{fejer}
L.~Fej\'er.
\newblock \"uber kreisgebiete, in denen eine wurzel einer algebraischen
  gleichung lieght.
\newblock {\em Jahresbericht der deutschen Math. Vereiningling}, 26:114--128,
  1917.

\bibitem{filaseta2}
M.~Filaseta, F.~Luca, P.~St\u{a}nic\u{a}, and R.~G. Underwood.
\newblock Two {D}iophantine approaches to the irreducibility of certain
  trinomials.
\newblock {\em Acta Arith.}, 128(2):149--156, 2007.

\bibitem{filaseta8}
M.~Filaseta, F.~Luca, P.~St\u{a}nic\u{a}, and R.~G. Underwood.
\newblock Galois groups of polynomials arising from circulant matrices.
\newblock {\em J. Number Theory}, 128(1):59--70, 2008.

\bibitem{charles}
Michael Filaseta, Carrie Finch, and Charles Nicol.
\newblock On three questions concerning {$0,1$}-polynomials.
\newblock {\em J. Th\'{e}or. Nombres Bordeaux}, 18(2):357--370, 2006.

\bibitem{filaseta3}
Michael Filaseta, Robert Murphy, and Andrew Vincent.
\newblock Computationally classifying polynomials with small {E}uclidean norm
  having reducible non-reciprocal parts.
\newblock In {\em Number theory week 2017}, volume 118 of {\em Banach Center
  Publ.}, pages 245--259. Polish Acad. Sci. Inst. Math., Warsaw, 2019.

\bibitem{flammang}
V.~Flammang.
\newblock The {M}ahler measure of trinomials of height 1.
\newblock {\em J. Aust. Math. Soc.}, 96(2):231--243, 2014.

\bibitem{fredricksen}
H.~Fredricksen and R.~Wisniewski.
\newblock On trinomials {$x^{n}+x^{2}+1$} and {$x^{8l\pm 3}+x^{k}+1$}\
  irreducible over {${\rm GF}(2)$}.
\newblock {\em Inform. and Control}, 50(1):58--63, 1981.

\bibitem{gallian}
{Joseph A.} Gallian.
\newblock {\em Contemporary abstract algebra}.
\newblock Cengage Learning, ninth edition edition, 2017.

\bibitem{greenfield}
Gary~R. Greenfield and Daniel Drucker.
\newblock On the discriminant of a trinomial.
\newblock {\em Linear Algebra Appl.}, 62:105--112, 1984.

\bibitem{grytczuk}
Aleksander Grytczuk and Jaroslaw Grytczuk.
\newblock Ljunggren's trinomials and matrix equation {$A^x+A^y=A^z$}.
\newblock {\em Tsukuba J. Math.}, 26(2):229--235, 2002.

\bibitem{gyory}
K.~Gy\H{o}ry and A.~Schinzel.
\newblock On a conjecture of {P}osner and {R}umsey.
\newblock {\em J. Number Theory}, 47(1):63--78, 1994.

\bibitem{hanson}
B.~Hanson, D.~Panario, and D.~Thomson.
\newblock Swan-like results for binomials and trinomials over finite fields of
  odd characteristic.
\newblock {\em Des. Codes Cryptogr.}, 61(3):273--283, 2011.

\bibitem{harrington}
Joshua Harrington.
\newblock On the factorization of the trinomials {$x^n+cx^{n-1}+d$}.
\newblock {\em Int. J. Number Theory}, 8(6):1513--1518, 2012.

\bibitem{herendi}
T.~Herendi and A.~Peth\"o.
\newblock Trinomials, which are divisible by quadratic polynomials.
\newblock {\em Acta. Acad. Paedagog. Agnensis, Sect. Mat. (N.S.)}, 22:61--73,
  1994.

\bibitem{hermez}
Alain Hermez and Alain Salinier.
\newblock Rational trinomials with the alternating group as {G}alois group.
\newblock {\em J. Number Theory}, 90(1):113--129, 2001.

\bibitem{hernane}
Mohand-Ouamar Hernane and Jean-Louis Nicolas.
\newblock Localisation de z\'{e}ros de familles de trin\^{o}mes.
\newblock {\em Ann. Fac. Sci. Toulouse Math. (6)}, 8(3):471--490, 1999.

\bibitem{hou}
Xiang-dong Hou.
\newblock A class of permutation trinomials over finite fields.
\newblock {\em Acta Arith.}, 162(1):51--64, 2014.

\bibitem{hou1}
Xiang-dong Hou.
\newblock Determination of a type of permutation trinomials over finite fields.
\newblock {\em Acta Arith.}, 166(3):253--278, 2014.

\bibitem{hou2}
Xiang-dong Hou.
\newblock A survey of permutation binomials and trinomials over finite fields.
\newblock In {\em Topics in finite fields}, volume 632 of {\em Contemp. Math.},
  pages 177--191. Amer. Math. Soc., Providence, RI, 2015.

\bibitem{hua}
Loo-Keng Hua and H.~S. Vandiver.
\newblock On the number of solutions of some trinomial equations in a finite
  field.
\newblock {\em Proc. Nat. Acad. Sci. U. S. A.}, 35:477--481, 1949.

\bibitem{jonassen}
Arne~T. Jonassen.
\newblock On the irreduciblity of the trinomials {$x^{m}\pm x^{n}\pm 4$}.
\newblock {\em Math. Scand.}, 21:177--189 (1969), 1967.

\bibitem{jones}
Lenny Jones and Tristan Phillips.
\newblock Infinite families of monogenic trinomials and their {G}alois groups.
\newblock {\em Internat. J. Math.}, 29(5):1850039, 11, 2018.

\bibitem{joyal}
A.~Joyal, G.~Labelle, and Q.~I. Rahman.
\newblock On the location of zeros of polynomials.
\newblock {\em Canad. Math. Bull}, 10:53--63, 1967.

\bibitem{kedlaya}
Kiran~S. Kedlaya.
\newblock A construction of polynomials with squarefree discriminants.
\newblock {\em Proc. Amer. Math. Soc.}, 140(9):3025--3033, 2012.

\bibitem{kelley}
Zander Kelley and Sean~W. Owen.
\newblock Estimating the number of roots of trinomials over finite fields.
\newblock {\em J. Symbolic Comput.}, 79(part 1):108--118, 2017.

\bibitem{kennedy}
E.~C. Kennedy.
\newblock Bounds for the roots of a trinomial equation.
\newblock {\em Amer. Math. Monthly}, 47:468--470, 1940.

\bibitem{knuth}
Donald~E. Knuth.
\newblock {\em The Art of Computer Programming, Volume 2 (3rd Ed.):
  Seminumerical Algorithms}.
\newblock Addison-Wesley Longman Publishing Co., Inc., Boston, MA, USA, 1997.

\bibitem{biswajit5}
Biswajit {Koley} and A.~{ Satyanarayana Reddy}.
\newblock Irreducibility of $x^n-a$.
\newblock {\em The Mathematics Student}, 89(1-2):169--174, 2020.

\bibitem{biswajit4}
Biswajit Koley and A.~Satyanarayana Reddy.
\newblock An irreducible class of polynomials over integers.
\newblock communicated.

\bibitem{biswajit2}
Biswajit {Koley} and A.~{Satyanarayana Reddy}.
\newblock {Irreducibility criterion for certain trinomials}.
\newblock {\em Malaya Journal of Mathematik}, S(1):116--119, 2019.

\bibitem{biswajit3}
Biswajit {Koley} and A.~{Satyanarayana Reddy}.
\newblock {An irreducibility criterion for polynomials over integers}.
\newblock {\em Bulletin math\'ematique de la Soci\'et\'e des Sciences
  Math\'ematiques de Roumanie}, 63(111)(1):83--89, 2020.

\bibitem{komatsu}
Kenz\^{o} Komatsu.
\newblock Integral bases in algebraic number fields.
\newblock {\em J. Reine Angew. Math.}, 278(279):137--144, 1975.

\bibitem{le}
Mao~Hua Le.
\newblock Irreducible quadratic factors of the trinomial {$x^n-x-a$}.
\newblock {\em J. Math. (Wuhan)}, 24(6):635--637, 2004.

\bibitem{lefton}
Phyllis Lefton.
\newblock A trinomial discriminant formula.
\newblock {\em Fibonacci Quart.}, 20(4):363--365, 1982.

\bibitem{lin}
Mu~Yuan Lin.
\newblock The irreducible quadratic factor of the trinomial {$x^n+x-a$}.
\newblock {\em Math. Appl. (Wuhan)}, 19(3):656--658, 2006.

\bibitem{liu}
Zhi~Wei Liu.
\newblock Irreducible quadratic factors of the trinomial {$x^n-bx+a$}.
\newblock {\em J. Math. (Wuhan)}, 27(6):684--686, 2007.

\bibitem{ljunggren}
Wilhelm Ljunggren.
\newblock On the irreducibility of certain trinomials and quadrinomials.
\newblock {\em Math. Scand.}, 8:65--70, 1960.

\bibitem{vila}
P.~Llorente, E.~Nart, and N.~Vila.
\newblock Discriminants of number fields defined by trinomials.
\newblock {\em Acta Arith.}, 43(4):367--373, 1984.

\bibitem{nart}
Pascual Llorente and Enric Nart.
\newblock Effective determination of the decomposition of the rational primes
  in a cubic field.
\newblock {\em Proc. Amer. Math. Soc.}, 87(4):579--585, 1983.

\bibitem{aaron}
Aaron Melman.
\newblock Geometry of trinomials.
\newblock {\em Pacific J. Math.}, 259(1):141--159, 2012.

\bibitem{JS}
J.~Mikusi\'{n}ski and A.~Schinzel.
\newblock Sur la r\'{e}ductibilit\'{e} de certains trin\^{o}mes.
\newblock {\em Acta Arith.}, 9:91--95, 1964.

\bibitem{anirban}
Anirban Mukhopadhyay, M.~Ram Murty, and Kotyada Srinivas.
\newblock Counting squarefree discriminants of trinomials under {$abc$}.
\newblock {\em Proc. Amer. Math. Soc.}, 137(10):3219--3226, 2009.

\bibitem{lagarias}
G.~Myerson.
\newblock \href{http://www.usenix.org/events/osdi04/tech/dean.html}{Western
  number theory problems, 16 \& 19 Dec 1999, Asilomar, CA}.
\newblock 1999.

\bibitem{nagell}
T.~Nagell.
\newblock Sur la r\`eductibilit\`e des trinomes.
\newblock {\em Attonde Skand. Mat.-Kongr. i Stockholm, 14-18 Augusti, 1934(=
  C.R. du huiti\'e congres des math\'ematiciens scandinaves tenu \`a Stockholm
  14-18 ao\^ut 1934), Lund}, 273--275, 1935.

\bibitem{AS12}
J.-L. Nicolas and A.~Schinzel.
\newblock Localisation des z\'{e}ros de polyn\^{o}mes intervenant en
  th\'{e}orie du signal.
\newblock In {\em Cinquante ans de polyn\^{o}mes ({P}aris, 1988)}, volume 1415
  of {\em Lecture Notes in Math.}, pages 167--179. Springer, Berlin, 1990.

\bibitem{ore}
\O. Ore.
\newblock Uber die reduzibilit\"{a}t von algebraischen gleichungen.
\newblock {\em Vid. selsk. Skrifter, Mat.-Nat. Kl., Kristiania}, (1):1--37,
  1923.

\bibitem{ore1}
Oystein Ore.
\newblock Contributions to the theory of finite fields.
\newblock {\em Trans. Amer. Math. Soc.}, 36(2):243--274, 1934.

\bibitem{osada}
Hiroyuki Osada.
\newblock The {G}alois groups of the polynomials {$X^n+aX^l+b$}.
\newblock {\em J. Number Theory}, 25(2):230--238, 1987.

\bibitem{osada1}
Hiroyuki Osada.
\newblock The {G}alois groups of the polynomials {$x^n+ax^s+b$}. {II}.
\newblock {\em Tohoku Math. J. (2)}, 39(3):437--445, 1987.

\bibitem{otake}
Shuichi Otake and Tony Shaska.
\newblock On the discriminant of certain quadrinomials.
\newblock In {\em Algebraic curves and their applications}, volume 724 of {\em
  Contemp. Math.}, pages 55--72. Amer. Math. Soc., Providence, RI, 2019.

\bibitem{LP}
Lauren\c{t}iu Panitopol and Doru \c{S}tef\u{a}nescu.
\newblock Some criteria for irreducibility of polynomials.
\newblock {\em Bull. Math. Soc. Sci. Math. R. S. Roumanie (N.S.)},
  29(77)(1):69--74, 1985.

\bibitem{perron}
Oskar Perron.
\newblock Neue {K}riterien f\"{u}r die {I}rreduzibilit\"{a}t algebraischer
  {G}leichungen.
\newblock {\em J. Reine Angew. Math.}, 132:288--307, 1907.

\bibitem{posner}
Edward~C. Posner and Howard Rumsey, Jr.
\newblock Polynomials that divide infinitely many trinomials.
\newblock {\em Michigan Math. J.}, 12:339--348, 1965.

\bibitem{rabinowitz}
Stanley Rabinowitz.
\newblock The {F}actorization of {$x^5\pm x + n$}.
\newblock {\em Math. Mag.}, 61(3):191--193, 1988.

\bibitem{rahman}
Q.~I. Rahman and G.~Schmeisser.
\newblock {\em Analytic theory of polynomials}, volume~26 of {\em London
  Mathematical Society Monographs. New Series}.
\newblock The Clarendon Press, Oxford University Press, Oxford, 2002.

\bibitem{redei}
L.~R\'{e}dei.
\newblock Ein {B}eitrag zum {P}roblem der {F}aktorisation von endlichen
  {A}belschen {G}ruppen.
\newblock {\em Acta Math. Acad. Sci. Hungar.}, 1:197--207, 1950.

\bibitem{ribenboim}
P.~Ribenboim.
\newblock On the factorization of {$X^n-BX-A$}.
\newblock {\em Enseign. Math. (2)}, 37(3-4):191--200, 1991.

\bibitem{AS3}
A.~Schinzel.
\newblock Solution d'un probl\`eme de {K}. {Z}arankiewicz sur les suites de
  puissances cons\'{e}cutives de nombres irrationnels.
\newblock {\em Colloq. Math.}, 9:291--296, 1962.

\bibitem{AS}
A.~Schinzel.
\newblock On the reducibility of polynomials and in particular of trinomials.
\newblock {\em Acta Arith.}, 11:1--34, 1965.

\bibitem{AS8}
A.~Schinzel.
\newblock Reducibility of lacunary polynomials. {I}.
\newblock {\em Acta Arith.}, 16:123--159, 1969/1970.

\bibitem{AS4}
A.~Schinzel.
\newblock Reducibility of lacunary polynomials. 1969 {N}umber {T}heory
  {I}nstitute ({P}roc. {S}ympos. {P}ure {M}ath., {V}ol. {XX}, {S}tate {U}niv.
  {N}ew {Y}ork, {S}tony {B}rook, {N}.{Y}., 1969).
\newblock pages 135--149, 1971.

\bibitem{AS1}
Andrzej Schinzel.
\newblock {\em Selected topics on polynomials}.
\newblock University of Michigan Press, Ann Arbor, Mich., 1982.

\bibitem{AS9}
Andrzej Schinzel.
\newblock On reducible trinomials.
\newblock {\em Dissertationes Math. (Rozprawy Mat.)}, 329:83, 1993.

\bibitem{AS10}
Andrzej Schinzel.
\newblock On reducible trinomials. {II}.
\newblock {\em Publ. Math. Debrecen}, 56(3-4):575--608, 2000.
\newblock Dedicated to Professor K\'{a}lm\'{a}n Gy\H{o}ry on the occasion of
  his 60th birthday.

\bibitem{AS11}
Andrzej Schinzel.
\newblock On reducible trinomials. {III}.
\newblock {\em Period. Math. Hungar.}, 43(1-2):43--69, 2001.

\bibitem{AS13}
Andrzej Schinzel.
\newblock On reducible trinomials, {IV}.
\newblock {\em Publ. Math. Debrecen}, 79(3-4):707--727, 2011.

\bibitem{viola}
Hans~Peter Schlickewei and Carlo Viola.
\newblock Polynomials that divide many trinomials.
\newblock {\em Acta Arith.}, 78(3):267--273, 1997.

\bibitem{AS5}
A.~Scinzel.
\newblock Some unsolved problems on polynomials.
\newblock {\em Mat. Biblioteka}, 25:63--70, 1963.

\bibitem{selmer}
Ernst~S. Selmer.
\newblock On the irreducibility of certain trinomials.
\newblock {\em Math. Scand.}, 4:287--302, 1956.

\bibitem{serret}
Joseph-Alfred Serret.
\newblock {\em Cours d'alg\`ebre sup\'{e}rieure. {T}ome {II}}.
\newblock Les Grands Classiques Gauthier-Villars. [Gauthier-Villars Great
  Classics]. \'{E}ditions Jacques Gabay, Sceaux, 1992.
\newblock Reprint of the fourth (1879) edition.

\bibitem{igor1}
Igor~E. Shparlinski.
\newblock On quadratic fields generated by discriminants of irreducible
  trinomials.
\newblock {\em Proc. Amer. Math. Soc.}, 138(1):125--132, 2010.

\bibitem{igor}
Igor~E. Shparlinski.
\newblock Squarefree parts of discriminants of trinomials.
\newblock {\em Arch. Math. (Basel)}, 102(6):545--554, 2014.

\bibitem{spearman1}
Blair~K. Spearman and Kenneth~S. Williams.
\newblock Quartic trinomials with {G}alois groups {$A_4$} and {$V_4$}.
\newblock {\em Far East J. Math. Sci. (FJMS)}, 2(5):665--672, 2000.

\bibitem{spearman}
Blair~K. Spearman and Kenneth~S. Williams.
\newblock The discriminant of a dihedral quintic field defined by a trinomial
  {$X^5+aX+b$}.
\newblock {\em Canad. Math. Bull.}, 45(1):138--153, 2002.

\bibitem{swan}
Richard~G. Swan.
\newblock Factorization of polynomials over finite fields.
\newblock {\em Pacific J. Math.}, 12:1099--1106, 1962.

\bibitem{szego}
G.~Szeg\"{o}.
\newblock Bemerkungen zu einem {S}atz von {J}. {H}. {G}race \"{u}ber die
  {W}urzeln algebraischer {G}leichungen.
\newblock {\em Math. Z.}, 13(1):28--55, 1922.

\bibitem{tverberg}
Helge Tverberg.
\newblock On the irreducibility of the trinomials {$x^{n}\pm x^{m}\pm 1$}.
\newblock {\em Math. Scand.}, 8:121--126, 1960.

\bibitem{uchida}
K\^{o}ji Uchida.
\newblock Unramified extensions of quadratic number fields. {I}.
\newblock {\em T\^{o}hoku Math. J. (2)}, 22:138--141, 1970.

\bibitem{uchida1}
K\^{o}ji Uchida.
\newblock Unramified extensions of quadratic number fields. {II}.
\newblock {\em T\^{o}hoku Math. J. (2)}, 22:220--224, 1970.

\bibitem{vahlen}
K.~Th. Vahlen.
\newblock \"{U}ber reductible {B}inome.
\newblock {\em Acta Math.}, 19(1):195--198, 1895.

\bibitem{vandiver}
H.~S. Vandiver.
\newblock On the number of solutions of certain non-homogeneous trinomial
  equations in a finite field.
\newblock {\em Proc. Nat. Acad. Sci. U. S. A.}, 31:170--175, 1945.

\bibitem{vandiver2}
H.~S. Vandiver.
\newblock On some special trinomial equations in a finite field.
\newblock {\em Proc. Nat. Acad. Sci. U. S. A.}, 32:320--326, 1946.

\bibitem{vandiver1}
H.~S. Vandiver.
\newblock On the number of solutions of some general types of equations in a
  finite field.
\newblock {\em Proc. Nat. Acad. Sci. U. S. A.}, 32:47--52, 1946.

\bibitem{peter}
Peter Vassilev.
\newblock On one trinomial equation and its solution.
\newblock {\em Adv. Stud. Contemp. Math. (Kyungshang)}, 20(4):489--497, 2010.

\bibitem{vishne}
Uzi Vishne.
\newblock Factorization of trinomials over {G}alois fields of characteristic
  {$2$}.
\newblock {\em Finite Fields Appl.}, 3(4):370--377, 1997.

\bibitem{von1}
Joachim von~zur Gathen.
\newblock Irreducible trinomials over finite fields.
\newblock In {\em Proceedings of the 2001 {I}nternational {S}ymposium on
  {S}ymbolic and {A}lgebraic {C}omputation}, pages 332--336. ACM, New York,
  2001.

\bibitem{von}
Joachim von~zur Gathen.
\newblock Irreducible trinomials over finite fields.
\newblock {\em Math. Comp.}, 72(244):1987--2000, 2003.

\bibitem{yamamoto}
Yoshihiko Yamamoto.
\newblock On unramified {G}alois extensions of quadratic number fields.
\newblock {\em Osaka Math. J.}, 7:57--76, 1970.

\bibitem{yamamura}
Ken Yamamura.
\newblock On unramified {G}alois extensions of real quadratic number fields.
\newblock {\em Osaka J. Math.}, 23(2):471--478, 1986.

\bibitem{yangfu}
Hai Yang and Ruiqin Fu.
\newblock On trinomials having irreducible quadratic factors.
\newblock {\em Period. Math. Hungar.}, 69(2):149--158, 2014.

\end{thebibliography}

\end{document}